%% file: 080712-BoijSoderbergConj.tex
\documentclass[12pt, leqno]{article}
\usepackage{amsmath,amscd,amsthm,amssymb,amsxtra,latexsym,epsfig,epic,graphics}
\usepackage[matrix,arrow,curve]{xy}
\usepackage{graphicx}
\usepackage{diagrams}
\input preamble.tex

\def\BB{{\mathbb B}}

\def\QQ{{\mathbb Q}}
\def\facet{{\bf facet}}
\def\image{{\rm image}}
\def\cE{{\cal E}}
\def\cF{{\cal F}}
\def\cG{{\cal G}}

\def\cHom{{{\cal H}om}}

 \def\bs{{Boij-S\"oderberg{} }}
\def\CC{{{\mathbb B}^*}}

\makeatletter
\def\Ddots{\mathinner{\mkern1mu\raise\p@
\vbox{\kern7\p@\hbox{.}}\mkern2mu
\raise4\p@\hbox{.}\mkern2mu\raise7\p@\hbox{.}\mkern1mu}}
\makeatother

\date{last revised July 12, 2008}
\title{Betti Numbers of Graded Modules and\\
Cohomology of  Vector Bundles}
\author{David Eisenbud and Frank-Olaf Schreyer\\ \\
Dedicated to Mark Green,whose work connecting\\
Algebraic Geometry and Free Resolutions\\ has inspired
us for a quarter of a century,\\ on the occasion of his Sixtieth Birthday}

\begin{document}

\maketitle

\begin{abstract}

In a remarkable paper
Mats Boij and Jonas S\"oderberg  \cite{BS}
conjectured that the Betti table of a Cohen-Macaulay 
module over a polynomial ring
is a positive linear combination of Betti tables of
modules with pure resolutions. We prove
 a strengthened 
form of their Conjectures. Applications include a proof of
the Multiplicity Conjecture of Huneke and Srinivasan and
a proof of the convexity of a fan naturally associated to the
Young lattice.

With the same tools we show that
the cohomology
table of any vector bundle on projective space is a positive rational
linear combination of the cohomology tables of
what we call \emph{supernatural}
vector bundles. Using this result we give new bounds on the 
slope of a vector bundle in terms of its cohomology.

\end{abstract}

\section*{Introduction}
 
 The \emph{Betti table} of a graded module $M$ over a graded ring $S$
 is numerical data consisting of the minimal number of generators
 in each degree required for each syzygy module of $M$. Similarly, the 
 \emph{cohomology table} of a coherent sheaf $\cF$ on a projective
 variety is  numerical data of the dimension of each cohomology group
 of each twist of the sheaf. In this paper we characterize
  the Betti tables of
 Cohen-Macaulay graded modules 
  and the cohomology tables of vector bundles on
 projective spaces, up to rational multiples. 
 
 In each case, the characterization is in terms of certain distinguished tables: the
 Betti tables of
 \emph{pure resolutions of Cohen-Macaulay modules} 
 on the one hand and the  cohomology
 tables of \emph{supernatural vector bundles}
  on the other. We also prove the existence of these special objects.
 They turn out to be the extremal rays in the appropriate
 convex rational cones of Betti tables and cohomology tables. 
 Though these cones are not dual to one another in 
 the usual sense, we use certain supernatural bundles to define the supporting
 hyperplanes of the cone of Betti tables, and certain pure resolutions to define the supporting hyperplanes of the cone of cohomology tables. 
 
After the first version of
this paper was posted, Boij and Soederberg \cite{BS2}
showed
that the functionals we define here actually 
cut out the rational cone generated
by all Betti tables---not just those of Cohen-Macaulay modules, 
as in this paper. Using this, they prove
strong versions of the Mulitplicity Conjecture for all modules.

 We next describe the results more precisely. Throughout this paper
 we will work with the polynomial ring $S=K[x_1,\dots,x_n]$ over a fixed,
 arbitrary field, and with the projective space $\PP^{n-1}:=\PP^{n-1}_K$
 over that field. To simplify notation we often write $m:=n-1$.
 All modules will be finitely generated graded $S$-modules with
 maps homogeneous of degree zero.
 All sheaves on $\PP^m$ will be coherent.
 
 \subsection*{Betti Tables}

Let $\BB$ be the rational vector space of \emph{rational Betti tables}: that is,
$\BB=\oplus_{-\infty}^\infty\QQ^{n+1}$, which we think of of as the space of 
tables of rational numbers
with $n+1$ columns and rows numbered by the integers.
To any length $n$ complex of finitely generated free $S$-modules
$$ 
F:\quad \oplus_j S(-j)^{\beta_{0,j}}\leftarrow \cdots \leftarrow \oplus_j S(-j)^{\beta_{i,j}}\leftarrow \cdots \leftarrow \oplus_jS(-j)^{\beta_{n,j}}\leftarrow 0
$$
we associate the \emph{Betti table\/} $\beta(F)\in\BB$ 
whose entry in the $i$-th column and $j$-th row is
$\beta_{i,i+j}$: 
\small
$$
\begin{matrix} 
j\backslash i&\vline&0   &  1    & \cdots & n    \cr\hline
\vdots&\vline&\vdots&\vdots & \cdots    &\vdots     \cr 
       0&\vline&\beta_{0,0}&\beta_{1,1}&\cdots&\beta_{n,n}\cr
       1&\vline&\beta_{0,1}&\beta_{1,2}&\cdots&\beta_{n,n+1}\cr
\vdots&\vline&\vdots&\vdots & \cdots    &\vdots     \cr 
\end{matrix}
$$         
\normalsize        
Finally, to any graded
$S$-module $M$ we associate the Betti table $\beta(M)$ of its
minimal free resolution. Note that
the direct sum of modules or resolutions corresponds to 
addition of Betti tables.

By a \emph{degree sequence} (of length $c$) we will mean a strictly increasing sequence of
integers $d=(d_0<\cdots<d_c)$. The resolution $F$ is called \emph{pure}, with
degree sequence $d$, if $\beta_{i,j}=0$ except when $j=d_i$. In this case
Herzog and K\"uhl \cite{HK} show that
$$
\beta_{i,d_i} = \lambda \prod_{j\neq i} \frac{1}{\mid d_j-d_i\mid} \quad\hbox{ for $0\leq i\leq c$}
$$
for some rational number $\lambda$.
The proof relies on the equations imposed on the $\beta_{i,j}(M)$ by the vanishing
of the first $c$ coefficients of the Hilbert polynomial of $M$, corresponding to the
fact that the support of $M$ has codimension $c$. We will call these the
\emph{Herzog-K\"uhl equations.}

The first of the \bs Conjectures is the existence
of Cohen-Macaulay modules with pure resolutions having
any given degree sequence. They prove this in the 
case $n=2$. It was proved for all $n$ in characteristic zero
 by Eisenbud-Fl\o ystad-Weyman \cite{EFW}. We prove it in general.
 
\begin{theorem}
 \label{existence0}
%begin{enumerate}
Every pure Betti table corresponding to a degree sequence
 of length $c\leq n$ is a rational multiple of the
Betti table of the minimal free resolution of a Cohen-Macaulay $S$-module.
\end{theorem}

We give the \emph{termwise} partial
order to the set of sequences of a given length,
$$
(d_0<\cdots<d_c) \leq (d'_0<\cdots<d'_c) \quad \Leftrightarrow \quad d_i\leq d'_i	\hbox{ for all }i.
$$
A totally ordered subset of a partially ordered set is called a \emph{chain}. 
The other parts of the \bs Conjectures may be summarized 
as follows.

\begin{theorem}\label{main0}
The Betti table of any finitely
generated graded Cohen-Macaulay $S$-module of codimension $c$
can be expressed uniquely as a positive
rational linear combination of the Betti tables of codimension $c$ Cohen-Macaulay modules
with pure resolutions whose degree sequences form a chain.
\end{theorem}

We strengthen this result to show that all Betti tables of the minimal resolutions
of all finitely generated graded $S$-modules, whether Cohen-Macaulay or not, 
lie inside a certain rational cone.

As noted in \bs \cite{BS}, the \emph{Multiplicity Conjecture} 
of Huneke and Srinivasan (see Herzog and Srinivasan \cite{HS}) follows from the conclusion
of Theorem \ref{main0}. Indeed,
the motivation of Boij and S\"oderberg in formulating their bold and
striking Conjectures  was apparently to prove the Multiplicity Conjecture.
In doing so, it seems to us, they provided a remarkable new way of approaching
free resolutions.

We denote by $e(M)$ the \emph{multiplicity}
of a module $M$; if $M$ is of finite length, a case to which all this reduces,
then $e(M)=\dim_K(M)$, the vector space dimension.

\begin{corollary} [Huneke-Srinivasan Multiplicity Conjecture]
If $M=S/I$ is a Cohen-Macaulay
algebra of codimension $c$, then
$$
\frac{1}{c!}\prod_i \min \{j\mid \beta_{i,j}\neq 0\}
\leq e(M)\leq 
\frac{1}{c!}\prod_i \max\{j\mid \beta_{i,j}\neq 0\}.\qed
$$
\end{corollary}

One can include all Cohen-Macaulay modules generated in degree
zero if one replaces $e(M)$ by a normalized multiplicity $e(M)/\beta_{0,0}(M)$.
See Francisco-Srinivasan \cite{Francisco-Srinivasan} for a
recent survey of the many papers on this Conjecture and its generalizations. 
For another application of Theorems \ref{main0} and \ref{existence0}
 see S\"oderberg \cite{JS}.

To prove Theorem \ref{main0} we construct the supporting hyperplanes
of the cone of positive rational multiples of Betti tables of Cohen-Macaulay
modules. These are defined in terms of certain free complexes that
come, in turn, from supernatural vector bundles,
which are vector bundles with special cohomology tables.

 \subsection*{Cohomology Tables}

Let $\CC=\prod_{-\infty}^\infty \QQ^{n+1}$ be the dual of $\BB$.
We think of $\CC$ as the 
space of \emph{rational cohomology tables},
tables with $n+1$ rows
and columns numbered by the integers, defined as follows:
To any 
complex of graded free $S$-modules
$$
E: \quad 0\to E^0\to E^1\to\cdots\to E^{n}\to 0
$$
we write $\gamma_{i,d}(E)$
for the value of the 
Hilbert function of $H^i(E)$ in degree $d$. 

 For example, let $\cE$ be a vector bundle
on $\PP^{n-1}$, and write 
$\cE^*$ for the dual bundle $\cHom(\cE, \cO_{\PP^{n-1}})$.
If  $E$ is the dual of a free resolution of 
the graded module $\oplus_{d} H^0(\cE^*(d))$, then
$\gamma_{i,d}=h^i(\cE(d))$ for $i<n-1$ and $\gamma_{i,d} = 0$
for $i=n-1$ or $n$.
A modification of these examples, with carefully chosen vector
bundles, will be used to define  the supporting
hyperplanes needed for the proof of Theorem \ref{main0}---see
Proposition \ref{linear monad}.

In either case, the 
associated \emph{cohomology table} is the table in $\CC$ with
$\gamma_{i, d-i}$ in the $i$-th row and the $d$-th column:
\small
$$
\begin{matrix}
\cdots&\gamma_{m,-m-1}&\gamma_{m,-m}&\gamma_{m,-m+1}&\cdots&\vline&m\cr
&\vdots&\vdots&\vdots &&\vline&\vdots\cr
\cdots&\gamma_{1,-2}&\gamma_{1,-1}&\gamma_{1,0}&\cdots&\vline&1\cr
\cdots&\gamma_{0,-1}&\gamma_{0,0}&\gamma_{0,1}&\cdots&\vline&0& \cr
\hline
\cdots&-1&0&1&\cdots&\vline&d\backslash i \cr
\end{matrix}
$$
\normalsize

We make this choice of indexing so that the cohomology table
of a vector bundle $\cE$ coincides with
the Betti table of the \emph{Tate resolution} of $\cE$. 
This is a minimal, doubly
infinite exact free complex over the exterior algebra on $n$ generators that is
connected with Bernstein-Gel'fand-Gel'fand duality. It is studied in
Eisenbud-Fl\o ystad-Schreyer \cite{EFS} and Eisenbud-Schreyer \cite{ES-C}.
For consistency with the notation of that paper,
we number the rows from the bottom and the columns from left to right
as in the table above. Proposition \ref{decreasing} gives some general restrictions
on cohomology tables.

Following Hartshorne-Hirschowitz \cite{VB} we say that
a sheaf $\cF$ on $\PP^{n-1}=\PP^{n-1}_K$ has \emph{natural cohomology}
if, for each integer $d$, the cohomology $H^i(\cF(d))$ is nonzero for at most
one value of $i$. 
We will say that $\cF$ has \emph{supernatural cohomology} if, in addition,
the Hilbert polynomial $\chi(\cF(d))$
has distinct integral roots. In this case we define the 
\emph{root sequence} of $\cF$ to be the 
sequence of roots in decreasing order, $z_1>\cdots> z_{m}$.
If $m=n-1$, the case of primary interest
to us, then any sheaf with natural cohomology is locally free
(Remark \ref{Ezra}), so we will
generally speak of supernatural vector bundles. 

\begin{theorem}\label{existence0a}
Any strictly decreasing sequence
of $n-1$ integers is the root sequence of a
supernatural vector bundle on $\PP^{n-1}$. \qed
\end{theorem}

In characteristic zero one can also construct supernatural
bundles with every root sequence using
Bott's
vanishing theorem. This was explained to us by Jerzy Weyman.
The construction, made explicit in Eisenbud-Schreyer \cite{ES-C}
Theorem 5.6 and summarized below, is
to apply an appropriate Schur functor to the universal
rank $n-1$ quotient bundle on $\PP^{n-1}$. 
\goodbreak

The condition of supernatural cohomology
is very tight: 
the cohomology table of a supernatural sheaf 
is determined by its root sequence and its rank
(see Theorem \ref{supernatural}.)
For example, the cohomology table  of a supernatural rank 3 vector bundle on $\PP^3$ with root sequence $z=(3,-1,-4)$ must be 
\scriptsize%\small
$$
\begin{matrix}
&\cdots \cr
&\cdots \cr
&\cdots \cr
&\cdots \cr\hline
&\cr
\end{matrix}\quad
\begin{matrix}
 90&45&16&0&0& 0 & 0 &0 &0&0\cr
 0&0& 0&6&5&0&0&0 & 0  &0\cr
 0&0&0&0& 0 &6&10&9&0&0  \cr
0&0&0&0& 0 &0& 0 &0& 20& 54  \cr\hline
-4&&&-1& &&&3&& \cr
\end{matrix}\quad
\begin{matrix}\
\cdots &\vline& 3\cr
\cdots &\vline&2\cr
\cdots &\vline&1\cr
\cdots & \vline&0& \cr\hline
&\vline&d\backslash i&\cr
\end{matrix}\;
$$
\normalsize

\scriptsize%\small
$$
\begin{matrix}
&\cdots \cr
&\cdots \cr
&\cdots \cr
&\cdots \cr\hline
&\cr
\end{matrix}\quad
\begin{matrix}
 90&45&16&0&0& 0 & 0 &0 &0&0\cr
 0&0& 0&6&5&0&0&0 & 0  &0\cr
 0&0&0&0& 0 &6&10&9&0&0  \cr
0&0&0&0& 0 &0& 0 &0& 20& 54  \cr\hline
-4&&&-1& &&&3&& \cr
\end{matrix}\quad
\begin{matrix}\
\cdots &\vline& 3\cr
\cdots &\vline&2\cr
\cdots &\vline&1\cr
\cdots & \vline&0& \cr\hline
&\vline&d\backslash i&\cr
\end{matrix}\;
$$
\normalsize

Inside the rational vector space $\CC$ lies the set of actual
cohomology tables of vector bundles, and since the direct sum
of bundles corresponds to the addition of tables, it is natural
to consider the rational cone that this set 
generates. We will show in \S \ref{vector bundles} that this convex cone is
the union of the simplicial cones of the
\emph{fan of supernatural bundles}. The cones in this fan
 correspond to termwise
totally ordered sets of root sequences, and have as 
extremal rays the cohomology tables of supernatural bundles.
Moreover
the supporting hyperplanes of the cone of cohomology tables of
vector bundles are given by linear functionals defined from pure free resolutions
of modules of finite length, using the same construction as we use for the proof
of the \bs Conjectures.
 
 \begin{theorem}\label{main1}
The cohomology table of any vector
bundle on $\PP_{K}^{n-1}$
has a unique expression as a positive
rational linear combination of the supernatural
cohomology tables corresponding to a chain of root sequences. 
\end{theorem}

As a corollary, we obtain an analogue of the Multiplicity Conjecture in the vector bundle setting,
giving new bounds for the slope of a vector bundle in terms of its cohomology.
To state it, we say that the \emph{cohomology range} of a vector bundle $\cE$ on
$\PP^{m}$ is a pair of weakly decreasing sequences of integers
$$
 r(\cE)=(r_1\ge \cdots \ge r_{m+1}=-\infty ) <
 R(\cE) =( \infty =R_0\ge\cdots\ge R_{m}) 
$$
that can be characterized as
the termwise greatest and least weakly decreasing
sequences, respectively, for which
$H^i \cE(d-i) \not= 0 \Rightarrow r_{i+1} < d < R_i$ (see 
\S 7 for another characterization.)

\begin{corollary}\label{slope}
If $\cE$ is a vector bundle
on $\PP^m$, then the slope $\mu(\cE):= \deg \cE/\rank \cE$ satisfies
$$ 
\frac{-1}{m}\sum_{i=1}^{m} R_i(\cE)
\le 
\mu(\cE) 
\le 
\frac{-1}{m}\sum_{i=1}^{m} r_i(\cE).
$$
\end{corollary}

\subsection*{Bilinear Functionals of Resolutions and Complexes}

\noindent 
The heart of this paper is the analysis 
of certain bilinear forms on $ \BB\times \BB^*$, and their positivity
and vanishing properties. 
The central object is the form
$$
\langle B,C\rangle := \sum_{\{i,j,k\mid j\leq i\}} (-1)^{i-j}\beta_{i,k}\gamma_{j, -k}.
$$
where $B=(\beta_{i,i+d})$ is a Betti table and $C=(\gamma_{j, d-j})$ is a 
cohomology table.
In \S \ref{pos1} we prove that when
$C$ is the cohomology table of a free
complex $E$ and 
 $B$ is the Betti table of a free resolution $F$ over $S$
 then
$$
\langle B,C\rangle = \sum_j \chi(F_{\geq j}\otimes H^j(E))\quad \geq \quad 0.
$$
Here $\chi$ denotes the Euler characteristic of the degree 0 part of the complex, that is,
$$
\chi(F_{\geq j}\otimes H^j(E)) := \sum_{k\geq j}(-1)^{k-j}\dim_K\bigl((F_k\otimes H^j(E))_0\bigr).
$$
The functionals $\langle F,E\rangle=:\langle \beta(F), \gamma(E)\rangle$ do 
\emph{not} directly give the supporting hyperplanes we need for the proof of 
Theorem \ref{main0},
because the cone of minimal resolutions of Cohen-Macaulay modules does not
contain all non-minimal resolutions. In \S \ref{pos2} we construct
 a family of modifications 
$\langle B,C\rangle_{c,\tau}$ that are still non-negative when
$B$ is the Betti table of a minimal free resolution. Applied with
$E$ a supernatural vector bundle, or $F$ a pure free resolution,
these define
the functionals needed for the proofs of Theorems \ref{main0}
and \ref{main1}. 

\subsection*{The Set of Betti Tables}

Though Theorem \ref{main0} gives the rational cone of minimal free resolutions
of Cohen-Macaulay modules, it does not tell us which points in $\BB$ are in the monoid
of Betti tables of
actual resolutions. Eisenbud, Fl\o ystad and Weyman \cite{EFW} conjecture
that any sufficiently large integral point on a ray of the cone corresponding
to a pure resolution should be
a $\beta(M)$ (this is not true in general for rays
not corresponding to pure resolutions.) Daniel Erman \cite{Erman}
gives many additional restrictions that points in the rational cone satisfy if they
come from resolutions, and shows that 
if $d<D$ then the monoid of resolutions of 
Cohen-Macaulay modules generated in degree $\geq d$ and having regularity $\leq D$ is 
finitely generated. Of course one can ask similar questions about cohomology
tables.

The monoid of resolutions
differs in characteristic zero and in finite characteristic.
For example, take $n=5$ and let $M$ be a general Artinian Gorenstein factor ring of $S$
with Hilbert function $1,5,5,1$. Such rings can be obtained from
the homogeneous coordinate rings of canonical curves of genus 7 by factoring
out a regular sequence of linear forms.
In characteristic zero the general ring of this type has Betti table
\scriptsize
$$
\beta(M)=
\begin{matrix}
1&0&0&0&0&0 \cr
0&10&16&0&0&0 \cr
0&0&0&16&10&0 \cr
0&0&0&0&0&1 \cr
\end{matrix}
$$
\normalsize
as proven by Schreyer \cite{Schreyer}.
But in characteristic 2 Kunte \cite{kunte} shows that
no Cohen-Macaulay module has this Betti table.
(Modules of finite length with
twice this Betti table are easily constructed
experimentally.) The Betti table of the general canonical
curve of genus 7 in characteristic 2 was shown by 
Schreyer \cite{Schreyer} (see also  Mukai \cite{Mukai} and \cite{Mukai-unpublished}) to be
\scriptsize
$$
\beta(M)=
\begin{matrix}
1&0&0&0&0&0 \cr
0&10&16&1&0&0 \cr
0&0&1&16&10&0 \cr
0&0&0&0&0&1 \cr
\end{matrix}.
$$
\normalsize

\subsection*{The Structure of This Paper} 
In \S \ref{non-cancellation} we give the algorithm for
producing the decomposition of Theorem \ref{main0} suggested
by Boij and S\"oderberg, and some examples of how
it can be used. In \S \ref{BS explanation}
we review the construction of the
 \bs fan associated to pure Betti tables. 
 In section \S \ref{Facet Equations} 
we review the description of the exterior facets of this
fan discovered by \bs \cite{BS}, and
we explain an algorithmic
construction of the supporting hyperplane of 
an exterior facet.

Sections \ref{pos1} and \ref{pos2} describe the bilinear
forms that we use, together with their positivity and vanishing
properties.

Our analysis of sheaves with supernatural cohomology is carried out in section \S
\ref{supernatural section}, and the proof of Theorem \ref{main0} is completed
in \S \ref{main proof}. The proofs of Theorem \ref{main1} and Corollary \ref{slope}
are carried out in \S \ref{vector bundles}.

\subsection*{Thanks!}
We are grateful to Jerzy Weyman for showing us that Ulrich Sheaves
could be constructed from homogeneous bundles
and teaching us about their cohomology; and to 
Dan Grayson and Mike Stillman, for their tireless assistance with 
Macaulay 2. We would never have discovered the crucial ideas in \S \ref{pos1}
and \S \ref{pos2}
without the extensive family of examples provided by
their program \cite{M2}. We also
thank Daniel Erman and Steven Sam for interesting discussions and
computer experiments related to this paper, and Silvio Levy for his
expert and unstinting help with TeX. A group of young mathematicians
at the 2008 AMS Math Research Communities meeting in Snowbird, Utah,
read this paper carefully and made many suggestions about exposition,
which we have gratefully adopted.

\section{An Algorithm and an Example: Consecutive Betti Numbers}
\label{non-cancellation}

The correctness of the following
algorithm for decomposing a Betti table may
clarify the meaning of Theorem \ref{main0},
from which it follows easily.
It was conjectured
in \bs \cite{BS}.

\smallbreak
{\small
\noindent{\bf Decomposition Algorithm}\\
\emph{Input:} A graded Cohen-Macaulay $S$-module $M$
of codimension $c$.\\
\emph{Output:} A list of positive rational coefficients $r_i$ 
and pure Betti Tables $\alpha^{(i)}$ whose degree sequences form a chain, such
that $\beta=\sum r_i\alpha_i$.

\begin{enumerate}
\item BEGIN: Set $L$ equal to the empty list. Set $\beta:=\beta(M)$.
\item For $i=0,\dots, c$ let $d_i=\min\{ j\mid \beta_{i,j}\neq 0\}$.
Let $\alpha$ be a pure Betti table with degree sequence $d=(d_0,\dots,d_c).$ 
Let $r$ be the largest rational number such that $\beta':=\beta-r\alpha$ has
non-negative entries. 
\item Add $(r, \alpha)$ to the list $L$. If $\beta'=0$ then END. Otherwise, set
$\beta:=\beta'$, and go to step 2.
\end{enumerate}
\smallbreak
}
%{\it Proof of Correctness.}
%The decomposition $\beta(M)=\sum s_i\beta^{(i)}$ guaranteed by
%Theorem \ref{main0} cannot contain any $\beta^{(i)}$
%whose entries are nonzero in positions where the entries of $\beta(M)$ are zero.
%The degree sequence constructed in the first step of the algorithm is thus
%the unique
%minimal degree sequence that can occur among the $\beta^{(i)}$. 
%If it were not present in the decomposition, then at least
%one of the $\beta_{i,d_i}(M)$ could not equal to the corresponding term of
%$\sum s_i\beta^{(i)}$. It cannot occur with a smaller
%coefficient than the one computed in the Algorithm since otherwise it would still
%be needed to decompose $\beta'$. The algorithm terminates because the table
%$\beta'$ produced in step 2 has strictly fewer nonzero entries than $\beta$.
%\qed

Here are some examples showing how Theorem \ref{main0} 
gives bounds on consecutive Betti numbers, and limits the extent
of ``non-cancellation'' in minimal resolutions.

\begin{example} Let $B_x$ be the Betti table
\scriptsize
$$
B_x=\begin{matrix}
1&0&0&0&0&0 \cr
0&10&16&x&0&0 \cr
0&0&x&16&10&0 \cr
0&0&0&0&0&1 \cr
\end{matrix}.
$$
\normalsize
In characteristic 0, the homogeneous coordinate
ring of the general canonical curve
of genus 7 has resolution with Betti table $B_0$,
and it is known (Schreyer \cite{Schreyer}) that if the curve admits a realization as
a plane sextic with 3 nodes, then the resolution has Betti table $B_9$. 
 The Decomposition Algorithm above begins
 with a Betti table 
\scriptsize
  $$
B'=\begin{matrix}
5&0&0&0&0&0 \cr
0&60&128&90&0&0 \cr
0&0&0&0&20&0 \cr
0&0&0&0&0&3 \cr
\end{matrix}
$$
\normalsize
that corresponds to the pure degree
 sequence $0,2,3,4,6,8$.
This implies that if $B_x$ is  a rational multiple
of the Betti table of a Cohen-Macaulay module, then
 $x/16\leq 90/128$, that is $x\leq 11.25$, and
Theorem \ref{existence0} shows  that there is a module of finite length whose Betti table is
  a multiple of $B_{11}$. In fact, writing 
 $B''$ for the Betti table of the resolution dual to the one described by $B'$, the reader
 may check that 
 $$
 B_{11} = \frac{1}{45}B_0+\frac{11}{90}B'+\frac{11}{90}B''.
 $$
 We do not know whether
 there are actually modules of finite length having Betti table $B_{11}$.
 
Here is another example, which gives a sharp bound the linear strand of a resolution.
Using \bs \cite{BS2}, one could even drop the Cohen-Macaulay
condition, which is why we have written ``projective dimension'' instead
of ``codimension''.

\begin{corollary} \label{1-regular} Let $K$ be any field and
let $M$ be a graded Cohen-Macaulay module of projective dimension $\leq c$, 
 generated in degree $\geq 0$ over $K[x_1,\ldots, x_n]$.
If $\beta_{p+1,p+1}(M)=0$ then 
$$
{\beta_{p,p}}(M) \le  \frac{c+2-p}{2p}\beta_{p-1,p-1}(M).
$$
\end{corollary}
\begin{proof} Consider, a module with pure resolution having degree sequence
$d=(d_0,\dots,d_c)$, and the ratio of its betti numbers
$\beta_{p-1,d_{p-1}}/\beta_{p,d_p}$. 
One checks from the Herzog-K\"uhl formula that this ratio is
a monotonically increasing function of the partial degree sequence
 $(d_{p+1},\dots,d_c)$. The ratio given in the Corollary
is the one associated to the pure resolution with the smallest degree sequence
having $d_0=0, d_p=p, d_{p+1}>p+1$, namely,
$0,1,\dots,p,p+2,\dots,c+1$. \end{proof}
\end{example}

\section{The \bs Fan and its Facets}
\label{Facet Equations}
\label{BS explanation}

A central insight of Boij and S\"oderberg \cite{BS} is
the identification of a certain simplicial fan inside $\BB$.
We begin by describing it.

For each interval $a<b$ of degree sequences, we consider the {\it order complex}
of this interval, whose simplices are the totally ordered subsets (chains) in the 
interval. For example, if $a_0<b_0$ are integers, then
the interval between $(a_0,a_0+1,\dots, a_0+n)$ and 
$(b_0,b_0+1,\dots, b_0+n)$
 is the well-known Young poset associated to the Schubert cell decomposition
of the Grassmannian $G(n+1,b_0-a_0)$. 
Since every
maximal chain of degree sequences between $d$ and $e$ has
1+$\sum (d_i-e_i)$ elements, 
the order complex is equidimensional. 
In fact, the restriction to particular intervals $[a,b]$ in the Young poset is
unnecessary, since the fan in a larger interval restricts to the fans in 
smaller intervals.

Boij and S\"oderberg show that the Betti tables coming from any one chain
of degree sequences are linearly independent, and that the 
map taking a degree sequence to a sequence of Betti numbers
of the corresponding pure resolution provides a geometric realization
of the order complex as a \emph{simplicial fan}---that is, a collection
of simplicial cones intersecting along faces---in $\BB$. 
We will call this fan the \emph{\bs fan}.
They conjectured that the union of the cones in the \bs fan is the convex
cone of Betti tables of Cohen-Macaulay modules, and this is the content
of Theorems \ref{existence0} and \ref{main0}. 

We will call a facet of a simplicial cone in the \bs fan
``exterior'' if it is contained in a unique simplicial cone in the fan. 
The exterior facets are described as follows.

\begin{proposition}\label{outer facets} \bs \cite{BS}[Proposition 2.12.]
Fix an interval $[\underline d, \overline d ]$ in the poset of
degree sequences, let  $\Delta$ be a
maximal
chain of degree sequences between $\underline d$ and
$\overline d$, and let $\Phi = \Delta\setminus \{f\}$ 
be a facet of $\Delta$. The facet $\Phi$ is exterior iff one of the following holds
\begin{enumerate}
\item $\Phi$ is obtained from $\Delta$ by removing the minimal or maximal
element. In the former case the supporting hyperplane of $\Phi$ is
$\beta_{i, \underline d_i}=0$, where $i$ is the index where
$\underline d$ and $\underline d^+$, the next sequence in the chain, differ, and dually in the latter case; or
\item The degree sequences $f^-$ and $f^+$ immediately below and above
$f$ in $\Delta$ differ in exactly one position, say $f^-_k<f_k<f^+_k$. Since
$\Delta$ is maximal we must have $f^-_k=f_k-1, f^+_k=f_k+1$, and the equation 
of the supporting hyperplane is $\beta_{k, f_k}=0$; or
\item The degree sequences $f^-$ and $f^+$ immediately below and above
$f$ in $\Delta$ differ in exactly two adjacent positions, say
$\tau$ and $\tau+1$, and we have $f_{\tau}=f^-_\tau=f^+_\tau-1,f_{\tau+1} =  f^-_{\tau+1}+1=f^+_{\tau+1}$ and
$f_{\tau+1}=f_\tau+2$ (see Example \ref{Ex1} for a picture.)
\end{enumerate}
\end{proposition} 

We denote the facet described in part 3 by $\facet(f, \tau)$.
To show that one of these exterior facets of the \bs fan
is an exterior facet of the cone of Betti tables, it suffices to prove that a supporting hyperplane of $\Phi$  is positive on all
Betti tables of minimal free resolutions of modules of finite length.
Without the existence this condition implies only that the cone of Betti tables is contained in the fan.
The form of the equations for the supporting
hyperplanes of the exterior facets in parts 1 and 2 in Proposition \ref{outer facets}
makes this positivity obvious, so it suffices to treat facets of the form
$\facet(f,\tau)$ as in Part 3) of the Proposition. 

Since the \bs fan lies in a proper subspace of $\BB$, the
facet equation for a given facet is not unique---we may add any combination of the
Herzog-K\"uhl equations. However, we will show that
 there
are special linear functions defining the facets that are
non-negative on the Betti table of \emph{any} minimal free resolution.
The  individual $\beta_{i,j}$ satisfy this property for
facets of type 1 and 2 in Proposition \ref{outer facets}. We will exhibit
distinguished equations for the supporting hyperplane of any facet 
of the form $\facet(f,\tau)$, which we will call the upper and lower
equations. We will show that the upper equation is
non-negative on every \emph{minimal} free resolution.
It can be constructed algorithmically:

\begin{proposition} \label{upper equation}
Let $f^-<f<f^+$ be degree sequences as in part 3 of Proposition \ref{outer facets},
There is a unique hyperplane $U$ in $\BB$ that contains $\facet(f,\tau)$ and 
also all tables concentrated in degrees $\geq f^+$, which we call the \emph{upper hyperplane}
for $\facet(f,\tau)$. Similarly, there is a unique \emph{lower hyperplane} that contains
all tables $\leq f^-$. 
\end{proposition}

We will call the equation of the upper hyperplane the \emph{upper equation} of the facet.

\begin{proof}
Proposition 2.12 of \bs \cite{BS} shows that the only thing that really matters
about $\facet(f,\tau)$ is the pair of degree sequences $f^-,f^+$, in the sense that
within the linear space defined by the Herzog-K\"uhl equations, the same linear
functionals vanish on all facets not containing $f$ but containing  $f^-,f^+$. 
Thus  $\facet(f, \tau)$ contains 
all pure Betti diagrams coming from degree sequences
 $\le f^-$ or $\ge f^+$ in the partial order. 
Using this we can construct the equation of the upper hyperplane
inductively
as follows. 

We identify linear functions on $\BB$ 
with (possibly infinite) tables of integers with $n+1$ columns (See Example \ref{Ex1}). 
Choose a maximal chain in $\BB$ containing $f^-, f, f^+$. Using this chain,
we will construct the
table corresponding to the linear function vanishing on the facet, which
we call the facet equation.
We begin by putting zeros in all the entries of the table corresponding to
degree sequences that are  $\geq f^+$, that is, we put
a zero in column $i$ and each row with index $\geq f^+_i-i$ for each $i=0\dots n$. Suppose that
the Betti numbers corresponding to $f^-$ are $\beta_0,\dots, \beta_n$.
In column $\tau$, row $f^-_\tau-\tau$ we put $\beta_{\tau+1}$ and in column $\tau+1$,
row $f^-_{\tau+1}-(\tau+1)$ we put $-\beta_\tau$; this ensures that our functional
will vanish on the Betti tables corresponding to $f^-$ as well as on that
corresponding to $f^+$, and will be positive on that corresponding to $f$.

Whatever we put in the other entries of the table, this functional will vanish on every pure Betti
table with degree sequence $d \ge f^+$. We can now solve for the remaining
(infinitely many) coefficients inductively. We start with $f^-$ and at each stage,
 we choose the next smaller degree sequence in our chain.
The corresponding Betti table contains only one nonzero entry in the region that is not yet determined,
so the vanishing condition allows us to solve uniquely for the corresponding element
of the table we are constructing.

The same idea could be used with
the roles of $f^-$ and $f^+$ reversed to produce the equation of the lower hyperplane.
\end{proof}

The entries in rows $\le f^-_i-i$ in the table corresponding to the upper equation will
be rather complicated (in the upper equation of Example \ref{Ex1}, for instance, the
entries are proportional to values of the polynomial $(z-4)(z-3)z(z+6)(z+7)(z+9)$;
see the proof of Theorem \ref{main}.) However there is one easy pattern:

\begin{corollary}\label{diagonals}
Let $b_{i,j}$ be the coefficient of $\beta_{i,j}$ in the table representing the upper equation of 
a  $\facet(f,\tau)$. If $j<f_i$ then $b_{i+1,j}=-b_{i,j}$. 
\end{corollary}

The Corollary says that in a table representing an upper facet equation,
the diagonals from lower left to upper right have zeros in the positions below those
corresponding to $f$, while above $f$
 the entries are of the same absolute value but alternating sign. This behavior
can be seen clearly in Example \ref{Ex1}.

\begin{proof} The lower equation has all zeros above the positions corresponding to $f^-$.
The upper equation differs from it
by a linear combination of Herzog-K\"uhl 
equations. These equations, in the case of modules of finite length that
we are considering, reduce to the equations saying that the
 Hilbert polynomial $p_M$ of
a module $M$ with the given Betti table is identically zero. 
Hilbert himself expressed $p_M$ in terms of the Betti table as
$$
p_M(d)=\sum_j p_S(d-j) \sum_i (-1)^i \beta_{i,j}(M)
$$ 
where $p_S(t)={t+n-1\choose n-1}$, regarded as a polynomial in $t$,
is the Hilbert polynomial of $S$.  To say that this polynomial of degree
$n-1$ is identically zero is to say that it has $n$ consecutive zero values,
and this is expressed by the vanishing of linear functions given by
 tables whose entries along the diagonals have constant absolute value
and alternating signs.
\end{proof}

\begin{example}\label{Ex1} The
 $\facet((-4,-3, 0,2,4,6,7,9), 3)$
involves the chain of degree sequences
$
f^-=(-4,-3, 0,2,3,6,7,9),\,f=(-4,-3,0, 2,4,6,7,9)$ and $f^+=(-4,-3,0,3,4,6,7,9)
$
 indicated
 by the diagram
\scriptsize
$$
\begin{matrix}
\vdots&\vdots&\vdots&\vdots&\vdots& \vdots &\vdots& \vdots &\cr
\cdot &\cdot & \cdot  &\cdot & \cdot  & \cdot  &\cdot &\cdot &\cr
*&*&\cdot &\cdot & \cdot  & \cdot  &\cdot &\cdot &\cr
\cdot &\cdot &\cdot &\cdot & \cdot  & \cdot  &\cdot &\cdot &\cr
\cdot &\cdot &*&\cdot & \cdot  & \cdot  &\cdot &\cdot &\cr
\cdot &\cdot & \cdot &-&-&\cdot  & \cdot &\cdot &\cr
\cdot &\cdot & \cdot  &+&+&\cdot  &\cdot &\cdot  & \;  0\hbox{-th row} \cr
\cdot &\cdot & \cdot  &\cdot & \cdot &  *&*&\cdot &\cr
\cdot &\cdot & \cdot  &\cdot & \cdot  &  \cdot &\cdot &*&\cr
\cdot &\cdot & \cdot  &\cdot & \cdot  & \cdot  &\cdot &\cdot &\cr
\vdots&\vdots&\vdots&\vdots& \vdots &\vdots& \vdots &\vdots\cr
\end{matrix}
$$
\normalsize
where the $\cdot $ represents a zero entry and the $-,\,+$ and $*$ are in the 
positions where the Betti tables corresponding
to $f^-,f^+$ or both are nonzero.
The linear functional defining the upper hyperplane
of this facet,
 computed by the algorithm above, is given by the dot
 product with the matrix
\scriptsize
$$
\begin{matrix}
\vdots&\vdots&\vdots&\vdots&\vdots& \vdots &\vdots& \vdots &\cr
1755&-385&0& 0&66& -70 & 0&100&\cr
385&0& 0&-66& 70 & 0&-100&175&\cr
0^*&0^*&66& -70& 0&100&-175&189&\cr
0&0&70&0& -100 & 175 &-189&140&\cr
0&0&0^*&100& -175 & 189&-140&60&\cr
0&0& 0&175&-189&140& -60&0&\;  \cr
0&0& 0 &0^+&0^+&60 &0&0 & \cr
0&0& 0 &0& 0 &  0^*&0^*&44&\cr
0&0& 0 &0& 0 &  0&0&0^*&\cr
0&0& 0 &0& 0 &  0&0&0&\cr
\vdots&\vdots&\vdots&\vdots& \vdots &\vdots& \vdots& \vdots\cr
\end{matrix}
$$
\normalsize
The zeroes marked $0^*$ and $0^+$ and all zeroes below are forced by our construction. The remaing zeroes  are in diagonals containing a $0^*$; these are implied by
Corollary \ref{diagonals}.

\end{example}

\section{Linear Functions Non-Negative on Betti Tables of Free Resolutions}
\label{pos1}

In this section we introduce the bilinear form on Betti tables and 
cohomology tables that is the fundamental tool of this paper.

If $M$ is a finitely generated graded  $S=K[x_1,\dots,x_n]$-module,
we set
$h_\ell(M)=\dim(M_\ell)$, the value of the Hilbert function of $M$ at $\ell$.
 For any bounded complex $G$ of such modules, we set
$\chi(G)=\sum_i(-1)^i h_0(G_i)$.

Given a Betti table $\beta\in \BB$ and a cohomology table $\gamma\in \CC$
we define
$$
\langle \beta, \gamma\rangle = 
\sum_{\{i,j,k\mid j\leq i\}}
(-1)^{i-j}\beta_{i,k}\gamma_{j,-k}.
$$
If $E$ is a bounded complex of modules with cohomological 
indices increasing from 0,
$$
E: 0\to E^0 \to E^1\to \cdots 
$$
and $F$ is a free complex, then we write
$$
\langle \beta, E\rangle,\quad
\langle F, \gamma\rangle,\quad{\rm\ and\ }\quad
\langle F, E\rangle
$$
for
$$
\langle \beta, \gamma(E)\rangle,\quad
\langle \beta(F), \gamma\rangle,\quad{\rm\ and\ }\quad
\langle \beta(F), \gamma(E)\rangle,
$$ 
respectively.
If $F_i=\sum_kS(-k)^{\beta_{i,k}}$ then
$\sum_k\beta_{i,k}h_{-k}(H^j(E))= h_0(F_i\otimes H^j(E))$, so we 
can simplify the formula to
$$
\langle F,E\rangle = \sum_j \chi(F_{\geq j}\otimes H^j(E)).
$$
Thus the value of the functional $\langle -,E\rangle$ on $F$ is
given by the dot product of the Betti table of $F$ with the matrix
$$
\begin{matrix}
\vdots&&\vdots&&\vdots \\
h^0_0 && h^1_{1}-h^0_1 && h^2_{2}-h^1_{2}+h^0_2 && \cdots \\
h^0_1 && h^1_{2}-h^0_{2} && h^2_3-h^1_3+h^0_3 && \cdots \\
h^0_2 && h^1_3-h^0_3 && h^2_4-h^1_4+h^0_4                && \cdots\\
\vdots&&\vdots&&\vdots \\
\end{matrix}
$$
where $h^i_j$ denotes the dimension of the degree $j$ component of $H^i(E)$
and the entry $h^0_0$ is
in the $(0,0)$ place.

Since these definitions involve only the cohomology of $E$, we have 
$\langle\beta,  E\rangle=\langle \beta, H^*(E)\rangle$, where
$H^*(E)$ is interpreted as a complex having zero differential, and
similarly for $F$---we could set the differentials of $F$ equal to zero
without disturbing the definition.
The importance of the complex $E$ and the differentials in $F$
appears from the following, which is the main technical result of this paper.
\goodbreak

\begin{theorem}
\label{positivity}
If 
$$
E:\quad 0\to E^0\to\cdots\to E^n\to 0
$$
is a complex of graded free $S$-modules
 and
$F$ is a free resolution of a graded
$S$-module $M$ 
then 
\begin{enumerate}
\item $\langle F,E\rangle\geq 0$. 
\item Suppose in addition that $M$ and the modules
$H^j(E)$, for $j>0$ have finite length.
If
\begin{eqnarray*}
0&>&  \reg M+\reg E^0\hbox{, and}\\
0&>& \reg F_{j-1}+\reg H^j(E) \hbox{     for every $j>0$,}
\end{eqnarray*}
then
$\langle F,E\rangle = 0.$ 
\end{enumerate}
\end{theorem}

\begin{example}\label{why free complex} 
Suppose that
\begin{align*}
E: \quad \cdots \to 0\to &E^j \to 0\to\cdots\\
&j&
\end{align*}
is 0 except in cohomological degree $j$,
and $F$ is the free resolution of a module $M$.
\begin{enumerate}
\item  If $E^j$ is free, then
$$
\langle F,E\rangle=\chi(F_{\geq j}\otimes E^j) = h_0({\rm Syz}_j(M) \otimes E^j).
$$
It is obvious that $\langle F,E\rangle\geq 0$ in this case.

\item If $E^j$ is not free then 
$%$
\langle F,E\rangle 
$ %$
may be negative. For example,
if $j=0$, $E=E^j=K$, 
and $F$ is the Koszul complex resolving $K(1)$, 
then $\langle F,E\rangle = -\beta_{1,0}(K(1))= -n$. 
Thus we
cannot drop the condition that $E$ is a free complex
in Theorem \ref{positivity}. 

\item But if $j=n, \ E=E^j = K$, and $F$ is the minimal free resolution of $M$,
then 
$\langle F,E\rangle = \beta_{n,0}(M)\geq 0.$ This positivity could have been 
deduced from Theorem \ref{positivity}, since
in this case we can rewrite
$\langle F,E\rangle= \langle F,E'\rangle$, where $E'$ is a Koszul complex.
\end{enumerate}
\end{example}

\begin{proof}[Proof of Theorem \ref{positivity}]
Consider the double complex $F\otimes E$:
\scriptsize
$$
\xymatrix{ & 0 & 0 &  & 0 & \cr
0 &\ar[l] \ar[u] F_0\tensor E^n &  \ar[l] \ar[u] F_1\tensor E^n & \ar[l]\ldots & \ar[l] \ar[u]  F_n \tensor E^n & \ar[l]  0 \cr
& \ar[u] \vdots & \ar[u] \vdots &  & \ar[u] \vdots & \cr
0& \ar[l] \ar[u] F_0\tensor E^1 &  \ar[l] \ar[u] F_1\tensor E^1 &\ar[l] \ldots & \ar[l] \ar[u]  F_n \tensor E^1 & \ar[l]  0 \cr
0 &\ar[l] \ar[u] F_0\tensor E^0 &  \ar[l] \ar[u] F_1\tensor E^0 &\ar[l] \ldots & \ar[l] \ar[u]  F_n \tensor E^0 & \ar[l]  0 \cr
 & \ar[u] 0 & \ar[u] 0 &  & \ar[u] 0 & \cr}
$$
\normalsize
and the two spectral sequences converging to its total homology.
Since $F$ is a resolution, the spectral sequence beginning with
$H_{\rm vert}H_{\rm hor}(F\otimes E)$ degenerates at that point, and we see
that the homology of the total complex $F\otimes E$ is $H^*(M\otimes E)$. In particular, 
$$
H^0_{\rm tot}(F\otimes E)= \ker(M\otimes E^0\to M\otimes E^1)
$$
and $H^j_{\rm tot}(F\otimes E)=0$ for $j<0$.

Now consider the other spectral sequence of the double complex, 
$$
{\bf E}: \qquad ^1{\bf E}=H_{\rm vert}(F\otimes E) \Rightarrow H_{\rm tot}(E\otimes F).
$$
We work in cohomological indices, considering $F_\ell$ to be in cohomological degree $-\ell$.
To simplify notation, set $N^j:=H^j(E)$ and write $N$ for $\sum_j N_j$.

Because each $F_k$ is free, 
$\left( H_{\rm vert}(F\otimes E)\right)^{(j, -\ell)}= F_\ell\otimes N^j$
so the 
value of the functional $\langle F,E\rangle$ is the Euler characteristic of  
the truncation of  $H_{\rm vert}(F\otimes E)$ consisting of the terms $F_\ell\otimes N^j$ of total
cohomological degree $j-\ell\leq 0$. The differential from $F$ makes
this into a complex that we will call $(F\otimes N)^{\leq 0}$, 
as in the following diagram:
\scriptsize
\begin{diagram}
&&&&&&0 &\lTo&F_n \tensor N^n&\lTo& 0 \\
&&&&&\Ddots &&\Ddots \\
{\bf ^1E'}=(F\otimes N)^{\leq 0}&&&&\Ddots &&\Ddots &&\vDots\\
&&&0 &\lTo& F_1\tensor N^1 & \lTo & \cdots & F_n \tensor N^1&\lTo& 0 \\
&0 &\lTo& F_0\tensor N^0 & \lTo& F_1\tensor N^0&\lTo &\cdots & F_n \tensor N^0&\lTo& 0 
\end{diagram}
\normalsize
\goodbreak
The differentials in {\bf E}
go from the  diagonal of cohomological degree $-\ell-1$, represented on page one by
$\sum_\ell F_{i+\ell+1} \otimes N^i$ to the diagonal of cohomological degree $-\ell$, represented on page 1 by 
$\sum_\ell F_{i+\ell} \otimes N^i$. Thus {\bf E} induces a spectral sequence ${\bf E'}$ whose first page is 
${\bf ^1E'}=(F\otimes N)^{\leq 0}$ as in the diagram above. 
Since taking homology preserves the Euler characteristic,
we see that $\langle F,E\rangle$ is the Euler characteristic of the infinity term $^\infty{\bf E'}$.

Since $^1{\bf E'}$ is zero in strictly positive cohomological degrees,
so is $^\infty{\bf E'}$. Since the total
complex of $F\otimes E$ has no homology in negative cohomological degrees,
the same is true of $(F\otimes E)^{\leq 0},$ so
 $^\infty{\bf E'}$ will also be zero in strictly negative cohomological degrees.  Thus
 the value of the functional
$\langle F,E\rangle$ is simply the dimension of the degree zero part of $^\infty{\bf E'}$.
In particular, it is non-negative, proving part 1) of the Theorem.

To prove the part 2) of the Theorem,  we compare $ {\bf E'}$ with {\bf E}.
We write $^td'$ and $^td$ for the 
differentials of the two sequences on the $t$-th page. 
The second differential 
$^2d$ goes from  the kernel of $^1d$ to
$F\otimes N$ modulo the image of $^1d$. Thus there is a well-defined
submodule of $F\otimes N$ that we may denote by
 $(\image\ {}^1d)+(\image\ {}^2d)$ even though $\image\ {}^2d$ is not itself a
 submodule of $F\otimes N$. Similarly,  it makes sense to speak of the submodules
 $$
 \sum_{t\geq 1}{}\image\  {}^td
  \quad\subset \quad 
  \bigcap_{t\geq 1} \ker\ {}^td
   \quad\subset \quad
    F\otimes N.
 $$

To simplify the notation in this and the next proof  set
\begin{eqnarray*}
K^i_t &=& (\sum_\ell F_\ell \otimes N^{\ell+i})\cap \bigcap_{s\leq t} \ker\ {} ^sd \\
I^i_t &=& (\sum_\ell F_\ell \otimes N^{\ell+i})\cap \sum_{s\leq t} \image\ {}^sd.
\end{eqnarray*}
for $1\leq t\leq \infty$.
Further, let $K^i_0=\sum_\ell F_\ell \otimes N^{\ell+i}$ and
$I^i_0 = 0$.

The spectral sequence {\bf E} gives a filtration of 
$ H_{\rm tot}(F\otimes E)$ whose associated graded module is
 $
^\infty{\bf E} =K_0^\infty/I_0^\infty.
%\frac{ \bigcap_{t\geq 1} \ker\ {}^td}
%{\sum_{t\geq 1}{}\image\  {}^td}.
$
Since the differentials $^td$ coming into the cohomological degree 0 diagonal terms
of {\bf E} coincide with the $^td'$, we have
$
^\infty{\bf E'}=
%\frac
({ \sum_\ell F_\ell\otimes N^\ell})/
{I^0_\infty}.
$
Thus 
$$
\langle F,E\rangle =
h_0\bigl( ^\infty{\bf E'}
\bigr) = 
h_0\bigl(H^0_{\rm tot}(F\otimes E)\bigr)+
h_0\bigl(\frac{ \sum_\ell F_\ell\otimes N^\ell}
{K_\infty}).
$$

The module 
$
(\sum_\ell F_\ell\otimes N^\ell)/
{K_\infty} 
$
is filtered by the submodules $K_t/K_\infty$.
The module $\sum_\ell F_\ell\otimes N^{\ell+1}$ is filtered by the $I^1_t$, and
the differential $^{t+1}d$ identifies the factor
$K^0_t/K^0_{t+1}$ 
 with $I^1_{t+1}/I^1_t$. 
  As already noted, we have 
$H^0_{\rm tot}(F\otimes E) \subset M\otimes E^0$.

Together with part 1) of the Theorem, these observations yield
 $$
0\leq  \langle E,F\rangle \leq h_0(M\otimes E^0) + h_0(\sum_\ell F_\ell\otimes N^{\ell+1}).
 $$
Under the hypotheses of of part 2) of the Theorem, 
both $M\otimes E^0$ and $\sum_\ell F_\ell\otimes N^{\ell+1}$
are zero in degree 0,
proving the vanishing of $\langle E,F\rangle$ in this case.
\end{proof}

\section{Linear Functions Non-Negative on Betti Tables of Minimal Free Resolutions}
\label{pos2}
In this section we construct a family of bilinear functions on Betti tables
and cohomology tables that are
modifications of  $\langle  \beta, \gamma \rangle$. When $\gamma=\gamma(E)$ 
for a free complex $E$ and $\beta=\beta(F)$ for a minimal free
resolution, the result  is non-negative, but it can be negative
when $F$ is non-minimal. Since the construction is somewhat opaque, we
begin by explaining it from several points of view.

Recall from Section \ref{pos1} that 
$$
\langle \beta, \gamma\rangle = \sum_{\{i,j,k\mid j\leq i\}}
(-1)^{i-j}\beta_{i,k}\gamma_{j,-k}.
$$
Given a cohomological index $\tau$ and a degree bound $c$ we set
\begin{eqnarray*}
\langle \beta, \gamma\rangle_{c,\tau} &=&
\sum_{\{i,j,k\mid j\leq i \hbox{\scriptsize\ and } (j<\tau \hbox{\scriptsize\ or } j \leq i-2)\}}
(-1)^{i-j}\beta_{i,k}\gamma_{j,-k}
\cr
&+&
\sum_{\{i,j,k,\epsilon \mid 0\leq \epsilon\leq 1,\ j=\tau,\ i=j+\epsilon,\ k\leq c+\epsilon \}}
(-1)^{i-j}\beta_{i,k}\gamma_{j,-k}.
\end{eqnarray*}
%Equivalently,
%\begin{eqnarray*}
%\langle \beta, E\rangle_{c,\tau} &=&
%\sum_{i<\tau,k}\beta_{i,k}h_{-k}(H^i(E))+
%\sum_{k\leq c}\beta_{\tau,k}h_{-k}(H^\tau(E))\cr
%&-&\sum_{i<\tau,k}\beta_{i+1,k}h_{-k}(H^{i}(E))-
%\sum_{k\leq c+1}\beta_{\tau+1,k}h_{-k}(H^{\tau}(E))\cr
%&+&
%\sum_{i,k}\beta_{i,k}\sum_{j\leq i-2}(-1)^{i-j}h_{-k}(H^j(E)) \cr
%%&+&
%%\sum_{i>,k}\beta_{i,k}\sum_{j\leq i-2}(-1)^{i-j}h_{-k}(H^j(E))
%\end{eqnarray*}
%\begin{eqnarray*}
%\langle \beta, \gamma\rangle_{c,\tau} &=&
%\sum_{i<\tau,k}\beta_{i,k}\gamma_{i,-k}+
%\sum_{k\leq c}\beta_{\tau,k}\gamma_{\tau,-k}\cr
%&-&\sum_{i<\tau,k}\beta_{i+1,k}\gamma_{i,-k}-
%\sum_{k\leq c+1}\beta_{\tau+1,k}\gamma_{\tau,-k}\cr
%&+&
%\sum_{i,k}\beta_{i,k}\sum_{j\leq i-2}(-1)^{i-j}\gamma_{{j,-k}} \cr
%&+&
%\sum_{i>,k}\beta_{i,k}\sum_{j\leq i-2}(-1)^{i-j}h_{-k}(H^j(E))
%\end{eqnarray*}
 
It may help the reader to see the coefficient of each $\beta_{i,i+\ell}$ (the entry in column $i$, row $\ell$)
 in the functional $\langle \beta, \gamma, \rangle_{\tau,c}$
explicitly. The formula depends on which region of the Betti table the index $(i, \ell)$ falls into,  as follows:
 $$
 \begin{matrix}
 & \vline&i<\tau&&\tau \leq i\leq \tau+1 && \tau+1<i \cr
 \hline
 \ell \leq c-\tau& \vline&U && V' &&  W\cr
 \ell > c-\tau &\vline&U && V  && W
\end{matrix}
 $$
%Recall that in the row labelled $\ell$ and the column labelled $i$ the entry of the Betti table
%is $\beta_{i, i+\ell}.$
%For the $\beta_{i,i+\ell}$ that appear in the betti table in the regions labelled 
%$A$
%(that is, $i<\tau$) the coefficient is $\sum_{j\leq i} (-1)^{i-j} \gamma_{j, -i-\ell}$ as for the 
%unmodified functional $\langle \beta, \gamma, \rangle$. In the regions labelled $C$ (with $i>\tau+1$)  it is similarly
%simple, though the range of $j$ is slightly shifted to $\sum_{j\leq i-2} (-1)^{i-j} \gamma_{j, -i-\ell}$.
%In the region $B$ (with $\tau\leq i\leq \tau+1, \ell > c-\tau$) the coefficient of $\beta_{\tau,\tau+\ell}$ is 
%$\sum_{j<\tau}(-1)^{i-j}\gamma_{j, -i-\ell}$, while in the region $B'$ (with $\tau\leq i\leq \tau+1, \ell > c-\tau$) the restriction on $j$
%changes to $j\leq \tau$. 
%
Recall that in the row labelled $\ell$ and the column labelled $i$ the entry of the Betti table
is $\beta_{i, i+\ell}.$ 
The coefficient of $\beta_{i,i+\ell}$ in the unmodified functional
$\langle \beta, \gamma, \rangle$ is $\sum_{j\leq i} (-1)^{i-j} \gamma_{j, -i-\ell}$. This differs
from the coefficient in the modified functional  only
in the range of $j$ in the summation. Region by region, as in the diagram above,
the coefficient of $\beta_{i,i+\ell}$ is:
\begin{enumerate}
\item[$U:$] $\sum_{\{j\mid j\leq i\}} (-1)^{i-j} \gamma_{j, -i-\ell}$---as for the 
unmodified functional $\langle \beta, \gamma, \rangle$.
\item[$V:$]$\sum_{\{j\mid j\leq \tau-1\}}(-1)^{i-j}\gamma_{j, -i-\ell}$
\item [$V'$:] $\sum_{\{j\mid j\leq\tau\}}(-1)^{i-j}\gamma_{j, -i-\ell}$
\item[$W$:] $\sum_{\{j\mid j\leq i-2\}} (-1)^{i-j} \gamma_{j, -i-\ell}$
\end{enumerate}

We were originally led to consider the functionals
$\langle \beta(F), \gamma(E)\rangle_{c,\tau}$ by the inspection of the upper facet
equations in many examples. The reader may also understand the idea better
from Examples  \ref{Example2} and \ref{7.3}.

We can follow the modification of the functional through the proof of Theorem
\ref{positivity} as well. The formulas  say that, applied to a free resolution 
$F$, the value $\langle F,E \rangle_{c,\tau} := \langle \beta(F), \gamma(E)\rangle_{c,\tau}$
is constructed from $\langle F,E\rangle$ by changing the upper parts
of the two leftmost diagonals in the complex ${\bf ^1E'}$ in the proof ofTheorem \ref{positivity}.
The following
 diagram shows the location of the changes we make
  to obtain the new function when $n=7$, $\tau =3$, the case of Example \ref{Ex1}:
\scriptsize
$$
\begin{matrix}
0&0&0&0&0&0&0&Z\cr
0&0&0&0&0&0&Z&Z\cr
0&0&0&0&0&Z&Z&*\cr
0&0&0&0&Z&Z&*&*\cr
0&0&0&C&C&*&*&*& \; \tau \hbox{-th row}\cr
0&0&*&*&*&*&*&*\cr
0&*&*&*&*&*&*&*\cr
*&*&*&*&*&*&*&*\cr
 \end{matrix}
$$
\normalsize
The entries marked $*$ in this diagram correspond to terms that are equal to 
 $F_i\otimes N^j$ as before; the entries $Z$ are set equal to zero;
 and the entries $C$ are altered by a degree restriction. The more general construction
given in the proof of Theorem \ref{modified} provides a more
conceptual picture of what is going on. 

\begin{theorem}\label{modified}
If $\beta$ is the Betti diagram of a minimal free resolution of a
graded $S$-module, then
$\langle \beta, E\rangle_{c,\tau} \geq 0$.
\end{theorem}

\begin{proof}
To clarify the proof we formulate a more general result.
Suppose we are in the situation of Part (1) of Theorem \ref{positivity}, and 
that we are given submodules
$A'_\ell \subset F_\ell\otimes N^\ell$ and 
$B'_\ell \subset F_{\ell+1}\otimes N^\ell$
for each $\ell$. Set
\begin{eqnarray*}
A':=&\sum_\ell A'_\ell &\subset\  A:=\  \sum_\ell F_\ell\otimes N^\ell \cr
B':=&\sum_\ell B'_\ell  &\subset\  B:=\  \sum_\ell F_{\ell+1}\otimes N^\ell.
\end{eqnarray*}
Let $D$ denote the diagram of modules that
agrees with $^1{\bf E'}=(F\otimes N)^{\leq 0}$ except for the first two diagonals,
where the terms $F_\ell\otimes N^\ell$ and 
$F_{\ell+1}\otimes N^\ell$
have been replaced by the terms
$A'_\ell$ and
$B'_\ell$
respectively, so that $A$ and $B$ are replaced by $A'$ and $B'$ (and we
forget, for a moment, the maps). Set
$$
\chi (D):= \sum_{j\leq i} (-1)^{i-j} D_i^j = \langle F,E\rangle-h_0A+h_0A'+h_0B-h_0B'.
$$
Let $^td$ be the differentials of the spectral sequence
{\bf E}. With notation as in the proof of Theorem \ref{positivity},
the spectral sequence gives us sequences
of submodules $K^{-1}_t \subset B$ and $I^0_t \subset A$,
such that $^{t+1}d$ identifies $K^{-1}_t/K^{-1}_{t+1}$ with $I^0_{t+1}/I^0_t$.

\begin{lemma} \label{technical}
If 
$$
 {}^{t+1}d(B'\cap K^{-1}_t) \subset (A'+I^0_t)/I^0_t
$$ 
for all $t\geq 0$, then
$
\chi (D)\geq 0
$
\end{lemma}

\begin{proof}[Proof of Lemma \ref{technical}]
 Set 
\begin{eqnarray*}
\overline B &=& K^{-1}_\infty\oplus \bigoplus_{t\geq 0} \frac{K^{-1}_t}{K^{-1}_{t+1}}   \\
\overline A &=& \frac{\sum_\ell F_\ell \otimes N^\ell}{I^0_\infty}\oplus \bigoplus_{t\geq 0} \frac{I^0_{t+1}}{I^0_t}.
\end{eqnarray*}
Let $\alpha:\overline B \to \overline A$ be the map taking $K_\infty^{-1}$ to zero and
identifying $K^{-1}_t/K^{-1}_{t+1}$ with $I^0_{t+1}/I^0_t$ by means of $^{t+1}d$.

Our hypothesis shows that $\alpha$ induces a map of associated
graded modules $\overline{B'}\to \overline{A'}$ coming from the induced
filtrations. Thus we have a commutative diagram with exact rows and
columns
 $$
\xymatrix{    
0 & C \ar[l]         & \ar[l] \overline A          & \ar[l]_\alpha \overline  B       & \ar[l] K&\ar[l] 0 \cr
0 & C'\ar[l] \ar[u]  & \ar[l] \ar[u] \overline {A'}  &\ar[u] \ar[l] \overline {B'} & \ar[u]\ar[l] K'   &\ar[l] 0 \cr
    &                       &    0\ar[u]              &  0\ar[u]              &   0\ar[u]           &            \cr
}
 $$
where $C=(\sum_\ell F_\ell \otimes N^\ell)/I^0_\infty$,
$K=K^{-1}_\infty$, and $K',C'$ are the kernel and cokernel of the map induced by $\alpha$.
 The proof of Theorem \ref{positivity} shows that $\langle F,E\rangle=h_0C$,
 so $\chi D = h_0C-h_0\overline A +h_0\overline{A'} +h_0\overline B-h_0\overline{B'}$,
 and an easy diagram chase shows that this is equal to
 $(h_0K-h_0K')+h_0C' $. Since $K'\subset K$, this is non-negative as required.
 \end{proof}

Proof of  Theorem \ref{modified} continued: To obtain $\langle F, E\rangle_{c,\tau}$ we consider
$$
F_i^a=\begin{cases} 
F_i & \hbox{ if $i<\tau$} \cr
\sum_{j\le c} S(-j)^{\beta_{\tau,j}} & \hbox{ if $i=\tau$} \cr
0 & \hbox{ if $i>\tau$ }
\end{cases}
$$
and 
$$
F_i^b=
\begin{cases} 
F_i & \hbox{ if $i<\tau+1$} \cr
\sum_{j\le c+1} S(-j)^{\beta_{\tau+1,j}} & \hbox{ if $i=\tau+1$} \cr
0 & \hbox{ if $i>\tau+1$ }
\end{cases}.
$$
Let
$$
A'_\ell=F_\ell^a \tensor N^\ell \subset F_\ell
 \tensor N^\ell
 $$ 
 and 
 $$
 B'_\ell= F_{\ell+1}^b \tensor N^\ell\subset  F_{\ell+1}\tensor N^\ell.
 $$ 
 Since $F$ is minimal, the differential 
 of $F$ maps $F_{\tau+1}^b$ to $F_\tau^a$, and it follows that
 $A',B'$ satisfy the hypotheses of Lemma \ref{technical} for $^1d$. 
 For $t>1$ the hypotheses are trivially satisfied:  the map 
 $$
 ^{t+1}d:B'_\ell \cap K^{-1}_t\to F_{\ell-t+1}/(I^0_t\cap F_{\ell-t+1})
 $$
 has source = 0 for $\ell>\tau$, while for $\ell\leq \tau$ 
 $F_{\ell-t+1}=F^a_{\ell-t+1}$, so $A=A'$ in this component.
  Moreover the resulting diagram $D$ satisfies $\chi D=\langle F,E \rangle_{c,\tau}$.
 By Lemma \ref{technical} we have $\langle F,E\rangle_{c,\tau}\geq 0.$
 \end{proof}

 \begin{example}\label{Example2}
The upper and lower equation of the facet  $\facet((-1,0,2,3),1)$ have coefficients
as indicated in the following tables
\scriptsize
$$
U:\quad \begin{matrix}
\vdots&\vdots&\vdots&\vdots\cr
21&-12&5&0\cr
12&-5&0&3\cr
5&0&-3&4\cr
0^*&3&-4&3\cr
0&0^+&0^+&0^*&\cr
0&0&0&0\cr
0&0&0&0\cr
\vdots&\vdots&\vdots&\vdots\cr
\end{matrix}\qquad\qquad
L:\quad \begin{matrix}
\vdots&\vdots&\vdots&\vdots\cr
0&0&0&0\cr
0&0&0&0\cr
0&0&0&0\cr
0^*&0^-&0^-&0\cr
3&-4&3&0^*&\quad \hbox{ 0-th row}\cr
4&-3&0&5\cr
3&0&-5&12\cr
\vdots&\vdots&\vdots&\vdots\cr
\end{matrix}
$$
\normalsize
If $M$ is any graded module with this Betti table, and $p_M$ is its Hilbert polynomial,
then one computes $U(\beta(M))-L(\beta(M))=5p_M(1)-3p_M(2).$
Thus  we are lead to consider a complex 
$$E: \quad 0 \to E^0=S(1)^5 \to E^1=S(2)^3 \to 0.$$
The equation $\langle -, E\rangle$ obtained from this complex
 has coefficients
\scriptsize
$$
\begin{matrix}
\vdots&\vdots&\vdots&\vdots\cr
21&-12&5&0\cr
12&-5&0&3\cr
5&0&-3&4\cr
0^*&3&-4&3\cr
0&4&-3&0^*&\quad \hbox{ 0-th row}\cr
0&3&0&0\cr
0&0&0&0\cr
\vdots&\vdots&\vdots&\vdots\cr
\end{matrix}
$$
\normalsize
To obtain the desired facet equation, we take $\tau=1$, $c=0$. 
The functional $\langle -,E\rangle_{0,1}$ is the result of replacing
the lowest 4 and the lowest two 3s by zero, obtaining the table $U$.

Notice that $\langle F,E\rangle_{c,\tau}$ is negative for some non-minimal resolutions $F$.
For example, if we add to the minimal resolution  of $M=K(1)$ and a trival complex as follows
\begin{diagram}[small]
&0 &\leftarrow &S(1) &\leftarrow &S^3 &\leftarrow &S(-1)^3 &\leftarrow &S(-2) &\leftarrow &0\cr
+&0 &\leftarrow &0 &\leftarrow &S(-1)     &\leftarrow &S(-1)&\leftarrow &0 &\leftarrow &0
\end{diagram}
we get a non-minimal resolution $F$ with Betti table
\scriptsize
$$
\begin{matrix}
\vdots&\vdots&\vdots&\vdots\cr
0&0&0&0\cr
1&3&4&1&\cr
0&1&0&0&\quad \hbox{ 0-th row}\cr
0&0&0&0\cr
\vdots&\vdots&\vdots&\vdots\cr
\end{matrix}
$$
\normalsize
and we compute $\langle F,E\rangle_{0,1}=-4<0$.
 \end{example}

\section{The Existence of Pure Resolutions}

\begin{theorem}\label{existence1} Let $K$ be any field,
and let $d=(d_0 < \cdots <d_n)$ be a sequence of integers.
There exists  a graded $K[x_1,\ldots,x_n]$-module of finite length
with
$
\beta_0=\prod_{i=1}^n{d_i-d_0-1\choose d_i-d_{i-1}-1}
$
generators,
whose minmal free resolution is pure with degree sequence $d$.
\end{theorem}

In the case where $d_i-d_{i-1}=1$ for all but one value of $i$, 
the complexes we produce coincide with those of
Buchsbaum-Eisenbud \cite{BE} and Kirby \cite{Kirby}
(see Eisenbud \cite{Eis}, Appendix A 2.6 for an exposition).

Our construction (and that for supernatural sheaves given in Section \ref{supernatural section})
starts from a sequence of multilinear forms
on a product of projective spaces posessing a property that
we could deduce from a general position argument if the ground
field were assumed infinite. Since we do not wish to make that
assumption, we give a direct construction.

\begin{proposition}\label{elementary forms}  Let
$m_0,\dots, m_k$ be non-negative integers, and let
the homogeneous coordinates on $\PP^{m_j}$ 
be $x_0^{(j)},\cdots,x_{m_j}^{(j)}$. The multilinear forms 
$$
x_\ell = \sum_{\mu_0+\cdots +\mu_k= \ell}\quad \prod_{j=0}^k x^{(j)}_{\mu_j}\qquad \hbox{ for } \ell=0,\ldots,\sum_{j=0}^k m_j
$$
have no common zeros in
$$
\PP_K^{m_0} \times \PP_K^{m_1}\times \cdots \times \PP_K^{m_k}.
$$
\end{proposition}

\begin{proof}
We do induction on $M:=\sum_{j=0}^k m_j$.
The case $M=0$ is trivial. Suppose that $M>0$ and the $x_\ell$ all vanish at 
at a point $P\in \prod \PP^{m_j}$. In particular,
$$
x_M=\prod_{j=0}^{k} x^{(j)}_{m_j}(P) =0.
$$
We cannot have $x_{m_j}^{(j)}(P)=0$ for $j$ such that $m_j=0$, so
$x_{m_j}^{(j)}(P)=0$ for some $j$ with $m_j\geq 1$.
Write $\PP^{m_j-1}$ for the subspace of $\PP^{m_j}$ where $x_{m_j}^{(j)}$
vanishes. The forms $x_0,\dots,x_{M-1}$ restrict to the
corresponding set of  forms on 
$$
\PP^{m_0}\times\cdots \times \PP^{m_j-1} \times \cdots \times \PP^{m_k}
$$
and vanish at $P$ there, contradicting the inductive hypothesis.
\end{proof}

We will construct pure resolutions by pushing forward other
pure resolutions, starting with the Koszul complex. The following result
is the key to this process.
For any product $X_1 \times X_2$ with projections $p:X_1 \times X_2 \to X_1$ and $q:X_1 \times X_2 \to X_2$ and sheaves $\cL_i$ on $X_i$, we set
$$
\cL_1 \boxtimes \cL_2 := p^*\cL_1 \otimes q^* \cL_2.
$$

\begin{proposition}
\label{pushdown construction} 
Let $\cF$ be a sheaf on $X \times \PP^m$, and let $p:X\times \PP^m \to X$
be the projection. Suppose that $\cF$ has a resolution of the form
$$
\cG: 0 \to \cG_N \boxtimes \cO(-e_N) \to \cdots \to  \cG_0 \boxtimes \cO(-e_0) \to \cF \to 0
$$
with degrees $e_0< \cdots < e_N$. If this sequence
 contains the subsequence $(e_{k+1},\ldots,e_{k+m})=(1,2,\ldots, m)$ for some $k\ge -1$ then 
$$
R^\ell p_*\cF=0 \hbox{ for } \ell>0
$$
and $p_* \cF$ has a resolution on $X$ of the form
\begin{align}
 0 \to \cG_N \tensor H^m\cO(-e_N) &\to \cdots  \cr
  \to \cG_{k+m+1}& \tensor H^m \cO(-e_{k+m+1} ) \rTo^\phi \cG_{k} \tensor H^0 \cO(-e_{k}) \to \cr
&\qquad\qquad\qquad\qquad\qquad\qquad
 \cdots \to \cG_0 \tensor H^0 \cO(-e_0)
\end{align}
\end{proposition}

\begin{proof} We may suppose
$m>0$. From the numerical hypotheses we see that $e_i\leq 0$ for $i\leq k$
and $e_i\geq m+1$ for $i\geq k+m+1$.

Consider the spectral sequence
$$
E^{i,-j}_1=R^ip_* (\cG_j\boxtimes \cO(-e_j)) \Rightarrow R^{i-j} p_* \cF.
$$
By the projection formula, the terms of the $E_1$ page are
$$
R^ip_* (\cG_j\boxtimes \cO(-e_j))=
\begin{cases}
\cG_j\tensor H^m(\PP^m,\cO(-e_j)) &\hbox{ if } j\geq k+m+1 \hbox{ and } i=m \cr
\cG_j \tensor H^0(\PP^m, \cO(-e_j)) &\hbox{ if } j\le k \hbox{ and } i=0 \cr
0 & \hbox{ otherwise. }
\end{cases}
$$
Thus the spectral sequence degenerates to the complex (1), where
$\phi$ is a differential from the $m$-th page 
and the other maps are differentials from the first page.
In particular only terms  $E^{i,-j}_{\infty}$ with  $i\leq j$ can be nonzero. 
On the other hand, the terms $R^{i-j} p_* \cF$  can be nonzero only for  $i\ge j$. 
Hence the complex (1) is exact and resolves 
$\oplus_{i\ge 0} E^{i,-i}_\infty=E^{0,0}_\infty= p_* \cF$, 
while the higher direct images of $\cF$ vanish.
\end{proof}

\begin{proof}[Proof of Theorem \ref{existence1}]
To simplify the notation we may harmlessly assume that $d_0=0$.
Let $m_0=m=n-1$, and for $i=1,\dots,n$ set
$m_i=d_i-d_{i-1}-1$, and set  $M=\sum_{j=0}^k m_j=d_n-1$. Choose
$M+1$ homogenous forms of multidegree $(1,\ldots,1)$ without a common zero on
$$
\PP:=\PP^m \times \PP^{m_1}\times \cdots \times \PP^{m_n},
$$
such as the forms described in Proposition \ref{elementary forms}.
Let
\def\cK{{\mathcal K}}
$$
\cK: 0\to \cK_{M+1} \to \cdots \to \cK_0 \to 0
$$ 
be the tensor product of the Koszul complex of these forms on $\PP$ and the line bundle
$
\cO_\PP(0,0, d_1,\ldots,d_{n-1}),
$
so $\cK_i=\cO_\PP(-i, -i, \dots, d_{n-1}-i)^{{d_n\choose i}}$ for $i=0,\dots, d_n$.
Let 
$$ 
\pi: \PP^m \times \PP^{m_1}\times \cdots \times \PP^{m_n} \to \PP^m
$$
be the projection onto the first factor. 
The complex $\cK$ is exact because the forms have no common zero.
Hence ${\bf R}\pi_* (\cK) = 0.$

The choices of these particular twists can be understood from the following
table, which gives some of the important ones, starting from the left hand end
of $\mathcal K$:
\scriptsize
\setcounter{MaxMatrixCols}{20}
$$
\begin{matrix}
\hbox{twist in:}&\vline&&&&&&&&\cr
\hline
\hbox {\rm first factor}&\vline&    -d_n &\hdotsfor 3& -d_{n-1}&\hdotsfor 3& -d_{n-2}&\cdots \cr
\hbox {last factor}&\vline& -m_n-1& -m_n &\cdots& -1&\hdotsfor 6\cr
\hbox {next-to-last}&\vline&\hdotsfor 4& -m_{n-1}-1&-m_{n-1}&\cdots&-1&\hdotsfor 2\cr
\cdots&\vline&\hdotsfor 8&-m_{n-1}-1&\cdots
\end{matrix}
$$
\setcounter{MaxMatrixCols}{10}
\normalsize
If we think of $\cK$ as a resolution of the zero sheaf $\cF=0$, and factor $\pi$ into
the successive projections along the factors of the product
$ 
 \PP^{m_1}\times \cdots \times \PP^{m_n},
$
then we may use Proposition \ref{pushdown construction} repeatedly 
to get a resolution of $\pi_* \cF = 0$ that has the form
$$ 
0 \to \cO^{\beta_n}(-d_n) \to \cdots \to\cO^{\beta_1}(-d_1)\to \cO^{\beta_0}.
$$
Taking global sections in all twists, we
get a complex
$$
0 \to S^{\beta n}(-d_n) \to \cdots \to S^{\beta_1}(-d_1) \to S^{\beta_0}
$$
that has homology of finite length. Since the
length of this complex is only $n$, the Lemme d'Acyclicit\'e of 
Peskine and Szpiro \cite{PS} (see also Buchsbaum and Eisenbud \cite{WMACE})
shows that the complex is actually acyclic.
Thus it is a pure minimal resolution, with the desired degree sequence,
of a graded module of finite length. 

Following the projections step by step, we see that
$$
\beta_0 = \prod_{i=1}^n h^0(\PP^{m_i}, \cO_{\PP^{m_i}}(d_{i-1}))= 
\prod_{i=1}^n{d_i-1\choose d_i-d_{i-1}-1}.
$$
Taking into account that a shift of all the $d_i$ by $d_0$ does not affect $\beta_0$,
we get the formula for $\beta_0$ in the Theorem.
\end{proof}

  \section{Sheaves with Supernatural Cohomology}
 \label{supernatural section}
 
In this section we work with coherent sheaves on a projective space $\PP^m$.
We begin by proving the existence of supernatural vector bundles, Theorem \ref{existence0a},
in a slightly sharper form.

\begin{theorem} \label{supernaturals exist} Let $K$ be any field. Suppose
that $m=\sum_{j=1}^k m_j $ with $m_j>0$.
If the sequence of integers  $z=(z_1>\cdots >z_m)$ consists of $k$ disjoint subsequences
of consecutive integers,  of lengths $m_1,\dots,m_k$,
then there exists a supernatural vector bundle $\cE$ on $\PP^m$  that is defined over $K$, 
has root sequence $z$, and has rank ${m \choose {m_1,\ldots, m_k}}$.
\end{theorem}

Here
${m \choose {m_1,\ldots, m_k}}$
denotes the \emph{multinomial coefficient}
$\frac{m!}{\prod_1^k m_j!}$

\begin{proof} Let $\nu_j$ denote the starting index of the $j$-th subsequence, so that
$$z_{\nu_j},\dots, z_{\nu_j+m_j-1}$$ 
are consecutive. 
Consider the  product 
$$
\PP^{m_1}\times\cdots \times \PP^{m_k}
$$
of $k$ projective spaces and the line bundle
$$
\cL=\cO(-z_{\nu_1}-1,\ldots,-z_{\nu_k}-1)
:=
p_1^*\cO(-z_{\nu_1}-1)\tensor\cdots\tensor p_k^*\cO(-z_{\nu_k}-1)
$$
on it. 

Let 
$$
\pi: \PP^{m_1}\times\cdots  \times \PP^{m_k} \to \PP^m
$$ 
be a finite morphism of degree
defined by linear projection from the 
Segre embedding of the product.  
For an explicit example showing that this can be defined over any field $K$,
 we can take  $\pi$ to be the map defined by the $m+1$ multilinear forms 
$$
x_\ell = \sum_{\mu_1+\cdots +\mu_k= \ell} \prod_{j=1}^k x^{(j)}_{\mu_j}\qquad \hbox{ for } \ell=0,\ldots,m,
$$
as in Proposition \ref{elementary forms}.

The desired supernatural bundle is 
$\cE=\pi_*\cL$. Indeed, since $\pi$ is a finite morphism we have
$$
H^i\cE(d)\cong H^i(\PP^{m_1}\times\cdots \times \PP^{m_k},\cL(d,\ldots,d)),
$$ 
which we can calculate from the the K\"unneth formula.
We see that $\cE$ has supernatural cohomology,  and its rank is
$$
\rank \cE= \deg \pi={m \choose {m_1,\ldots, m_k}}.\eqno{\qed}
$$ 
\renewcommand{\qedsymbol}{}
\end{proof}

Here is an alternate, characteristic zero construction of supernatural bundles, 
using representation theory.
It is restricted to characteristic zero by the
use of Bott's Vanishing Theorem, but we include it
because it yields bundles whose rank is often not an integral multiple of the ranks of 
the bundles produced in Theorem \ref{supernaturals exist}.
 
\begin{theorem}\label{supernaturals in char 0 exist} 
Let $K$ be a field of characteristic 0. If $z=(z_1>\cdots>z_{m})$
is a sequence of integers, then there exists a $GL_{m+1}$-equivariant vector
bundle $\cE$ on $\PP_K^m$  with supernatural cohomology
and root sequence $z$. 
\end{theorem}

\begin{proof} 
The desired vector bundles, up to a twist, can be constructed by 
applying Schur functors to the tautological rank $m$ quotient bundle $Q$
on $\PP^m$: Bott's Theorem (see for example Weyman \cite{W}) says that if 
$1\le i \le m-1$ then the cohomology group 
$H^i((S_\lambda Q)(d))$ is nonzero if and only if  
$\lambda_{m-i+1} < -d-i \le \lambda_{m-i}$, while
 $H^0 S_\lambda Q (d)=0$ if and only if $d<0$ and 
$H^m S_\lambda Q(d)=0$ if and only if $m\ge -m-\lambda_1-1$. 

To obtain a desired bundle with supernatural cohomology as in the Theorem  we choose
$$
\lambda_i=z_1-z_{m+1-i}-m+i \quad {\rm\ for\ }\quad 1\leq i\leq m-1,
$$
and take $\cE=(S_\lambda Q)(-z_1-1)$.
\end{proof}

This result has an appealingly transparent statement
in terms of Betti diagrams of Tate resolutions.
For a partition $\lambda=(\lambda_1 \ge \lambda_2\ge \cdots \ge \lambda_{m-1})$, say $\lambda=(8,7,7,2,0)=(8,7,7,2)$,
%m=6 and
%facet((-4,-3,0,2,4,6,7,9),3)  =>( z_1,...,z_6)=(4,3,0,-6,-7,-9) =>
% \lambda=(4+9-6+1,4+7-6+2,4+6-6+3,4+0-6+4,4-3-6+5)=(8,7,7,2,0)
 which corresponds to Example \ref{Ex1}, the Tate resolution of the homogeneous bundle $S_\lambda Q$ 
has nonzero terms only in the degrees marked $*$
in the following Betti diagram, in which the Ferrers diagram is indicated.

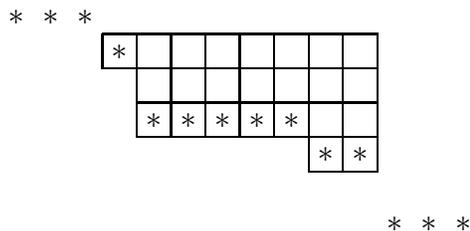
\centerline{
\vbox{
\def\star{\rlap{\hbox to 13pt{\hfil\raise3.5pt\hbox{$*$}\hfil}}}
\def\ {\hbox to 13pt{\vbox to 13pt{}\hfil}}
\def\*{\star\ }
\def\_{\hbox to 13pt{\hskip-.2pt\vrule\hss\vbox to 13pt{\vskip-.2pt
            \hrule width 13.4pt\vfill\hrule\vskip-.2pt}\hss\vrule\hskip-.2pt}}
\def\x{\star\_}
\offinterlineskip
\hbox{\*\*\*\ \ \ \ \ }
\hbox{\ \ \   \x\_\_\_\_\_\_\_\ \ \ }
\hbox{\ \ \   \  \_\_\_\_\_\_\_\ \ \ }
\hbox{\ \ \   \  \x\x\x\x\x \_\_\ \ \ }
\hbox{\ \ \   \  \  \  \  \  \  \x\x \ \ \ }
\hbox{\ \ \   \  \  \  \  \  \  \  \  \ \ \ }
\hbox{\ \ \   \   \  \  \  \  \  \  \  \*\*\*}
}
}  
\noindent 
See Theorem 5.6  in Eisenbud-Schreyer \cite{ES-C}.

We want to characterize the cohomology tables of supernatural
bundles, and to this end we next record a basic fact about cohomology tables:

\begin{proposition} \label{decreasing} If $\cE$ is a nonzero coherent sheaf
on $\PP^{m}$ then
every column of the cohomology table of $\cE$ contains a nonzero entry; that is, 
for each integer $d$ some $\gamma_{i,i-d}(\cE) := h^i\cE(d-i) \neq 0$.
Furthermore,
\begin{align*}
 d\mapsto M_{d}&:=\max\{i \mid \gamma_{i,d-i}\not=0 \} \ {\rm and  }\cr
 d\mapsto m_{d}&:=\min\{i \mid \gamma_{i,d-i}\neq 0\} 
\end{align*}
are weakly decreasing functions of $d$.
\end{proposition}

\begin{proof} \def\TT{{\bf T}} The Tate resolution $\TT(\cE)$ is a minimal free doubly 
infinite exact complex,
over the exterior algebra $\Lambda$ on $n$ generators of degree $-1$. 
By Eisenbud-Fl\o ystad-Schreyer \cite{EFS},Theorem 4.1,
the term of $\TT(\cE)$ having cohomological degree $d$ is
$$
\bigoplus_j H^j(\cE(d-j))\otimes \Lambda(j-d).
$$
Moreover, since $\Lambda$ is
self-injective, the dual of $\TT$ is also a resolution. So no term of  $\TT(\cE)$ can be zero, which gives the first statement.
Since the generators of
the exterior algebra are of negative degree, there are nonzero
maps $\Lambda(-d)\to\Lambda(-e)$ only if $d<e$. Since $T(\cE)$ is
a minimal complex, it cannot have any maps of degree 0.
Hence  $M_d$ is  weakly decreasing.
For $m_d $ we apply the same argument to the dual of $\TT(\cE)$ and obtain that $-m_{-d}$ is a weakly decreasing function. Hence $m_d$ is a weakly decreasing function as well.
\end{proof}

By the \emph{rank} of a coherent sheaf $\cE$ on $\PP^m$ we mean the normalized leading
coefficient of the Hilbert polynomial,
$$
\chi(\cE(d))= \frac{\rank \cE}{s!} d^{s} + O(d^{s-1}).
$$
where $s=\dim \cE$.

\begin{theorem}\label{supernatural}
If $\cE$ is a supernatural sheaf of dimension $s$ with root sequence
 $z_1>\cdots >z_{s}$, and we set $z_{0} = \infty$ and $z_{s+1}=-\infty$ then,
 for each $0\leq j\leq s$,
 $$
 h^j\cE(d)=
 \begin{cases}
  \frac{\rank \cE}{s!} \prod_{i=1}^{s}\mid d-z_i\mid &\text{if $z_{j}>d> z_{j+1}$},\\
 0 & \text{otherwise}.
 \end{cases}
$$
\end{theorem}

\begin{proof} Define $M_{d}, m_{d}$ as in Proposition \ref{decreasing}.
Of course $M_d\geq m_d\geq 0$, and by 
Serre's Vanishing Theorem we have
$M_d=m_d=0$ for $d\gg 0$.
Since the Hilbert polynomial of $\cE$ has $s$ zeros,
the dimension of the support of $\cE$ is $s$. It follows that 
$H^j\cE(d)=0$ for $j>s$ and any $d$, while if $d\ll0$ then
$
H^s\cE(d)\neq0.
$.
Thus $M_d=s$ for $d\ll 0$.   Since $\cE$ has natural cohomology this implies $m_d=s$ for $d\ll 0$ as well. By Proposition \ref{decreasing}, 
$M_d$ and $m_d$ are weakly decreasing from $s$ to 0.

Since the $M_d$ are weakly decreasing,
 the sequence
of numbers $d-M_d$ is strictly increasing. It
omits precisely those values $z$ such
that $z=d-i$ with $M_d>i\geq M_{d+1}$. This
means that precisely $s$ distinct values are omitted
from the sequence $d-M_d$. Exactly the same
considerations apply to the sequence $d-m_d$.

If $\cE$ has natural cohomology, then the vanishing
of $\chi(\cE(z))$ implies the vanishing of all $H^j(\cE(z))$,
so the integral roots of the Hilbert polynomial must be among
the omitted values of the sequences $\{d-M_d\}$ and $\{d-m_d\}$.
If $\cE$ has supernatural cohomology, then there are $s$ integral roots,
which thus 
give all the omitted values. It follows that the omitted values are the same for
$\{d-M_d\}$ and $\{d-m_d\}$. Since these two sequences
are the same for $d\ll 0$, they must be the same for all $d$; that is,
$M_d=m_d$ for all $d$.
Moreover, $M_{d+1}=M_d-k$ if and only if there are
exactly $k$ roots of the Hilbert polynomial between
$d-M_d$ and $d+1-M_{d+1}$. By induction we see
that the value of $M_d$
is equal to the number of roots above $d$.

The condition of natural cohomology implies that the value of
$|\chi(\cE(d))|$ is the value of some
$h^j\cE(d)$. The formulas above tell us the value of $j$. 
The zeros determine the Hilbert polynomial as
$\chi(\cE(d))= C\cdot \prod_{i=1}^s(d-z_i)$ for some constant $C$,
and $C$ can be computed  by comparing leading coefficients, yielding the formula
given. 
\end{proof}

\begin{remark} \label{Ezra}
A sheaf on $\cE$ on $\PP^m$ is a vector bundle if and only if the intermediate cohomology modules
$\oplus_d H^i \cE(d)$ for $1 \le i \le m-1$ have finite length. Thus a supernatural sheaf on $\PP^m$  of full dimension
 $\dim \cE = m$ is a vector bundle by Theorem \ref{supernatural}.
 \end{remark}
 
 The rank provided by the Theorem \ref{supernaturals exist} 
 is not always the smallest possible. For example, the root sequence
 $z=(2,0,-2)$ occurs for any
 null-correlation bundle on $\PP^3$ with monad
 $$
 0\to \cO_{\PP^3}(-3)\to\cO^4_{\PP^3}(-2)\to\cO_{\PP^3}(-1)\to0,
 $$
 and such a bundle has rank $2$, while Theorem \ref{supernatural} 
 provides a bundle of rank 6. A general lower bound for the rank, which gives 2 in this case,
  is provided by the following:
 
 \begin{proposition} \label{supernatural rank bound} Let $z=(z_1>\cdots > z_m)$ be a root sequence for supernatural bundles $\cE$ on $\PP^m$. For each prime $p$ let $e_p(z)$ denote the maximal integer such that the root sequence $z$ contains each residue mod $p$  at least $e_p(z)$ times, and let $c(z)=\prod_p p^{e_p(z)}$. Then the rank of a supernatural bundle with root sequence $z$ is a 
multiple of $\frac{m!}{c(z)}$.
\end{proposition}

\begin{proof} The polynomial $$\frac{1}{c(z)}\prod_1^m(t-z_j)$$ takes integral values which have no common factor.
\end{proof}

\begin{remark} \label{supernatural ranks} Note that rank $\frac{m!}{c(z)}$ is not always possible for a supernatural bundle  with root sequence $z$. A simple example is the root sequence $(2,1,-2,-3)$, which is not possible for a rank $2$ or even rank $4$ vector bundle on $\PP^4$. 
In fact, the Tate resolution for a supernatural bundle of rank $2r$ with this root 
sequence would be
\scriptsize
$$
\begin{matrix}
56r& 21r& 5r&0&0&0&0&0\\
0&0&0&0&0&0&0&0\\
0&0&0&r& r&0&0&0\\
0&0&0&0&0&0&0&0\\
0&0&0&0&0&5r& 21r& 56r
\end{matrix}
$$
\normalsize
However the entries of the 
$r\times r$ matrix of linear forms must span all the linear forms on $\PP^4$,
so $r\geq 3$.
In fact, a generic choice of the 3x3 matrix leads to such bundles
with rank 6, which is the rank
provided by the proof of Theorem \ref{supernaturals exist}.
Moreover, Theorem \ref{supernaturals in char 0 exist} gives a bundle of this type having rank $20$ so that, 
at least in characteristic 0, every sufficiently large even number occurs as the rank 
of a supernatural bundle with root sequence $(2,1,-2,-3)$.

 For arbitrary root sequences $z$,
we conjecture that any sufficiently high multiple of $\frac{m!}{c(z)}$ actually occurs as a rank.\end{remark} 
  
\emph{Linear monads} for vector bundles, given by the next Proposition, will be central in our use of supernatural bundles.

\begin{proposition}\label{linear monad} Let $\cE$  be a vector bundle
on $\PP^{m}$, and let $a$ be an integer. If $\cE^*$ is $a$-regular then
there exists a linear complex 
$$0 \to E^0 \to E^1 \to \cdots \to E^{m} \to 0 $$
with $E^k =S(a+k)^{b_k}$, with homology
$H^ iE= \sum_d H^i\cE(d) \hbox{ for $i<m$}$ and
$H^{m} E=\sum_{d\ge -a-m} H^{m}\cE(d)$.
\end{proposition}

\begin{proof} Since  $\cE^*$ is $a$-regular, $\sum_{d\ge a} H^0\cE^*(d)$ has a linear resolution
$$0\leftarrow \sum_{d\ge a} H^0\cE^*(d)\leftarrow S(-a)^{b_0} \leftarrow \cdots \leftarrow S(-a-m)^{b_m} \leftarrow 0.$$ 
The dual  complex is the desired complex $E$. Its sheafification $\tilde E$ has homology $H^0(\tilde E) \cong \cE$ and is exact otherwise. The statement about the homology follows by chasing sheaf cohomology through the sheafified  complex. See Section 8 of 
Eisenbud-Fl\o ystad-Schreyer \cite{EFS} for more about linear monads. 
\end{proof}

Proposition \ref{linear monad} applies to vector bundles with supernatural cohomology:

\begin{proposition} \label{dual regularity}
Let $\cE$ on $\PP^{m}$ be a vector bundle with supernatural cohomology with zeroes $z_1 >  \cdots >z_{m}$ of the Hilbert polynomial. Let $a\ge -z_{m}-m$  be an integer. Then $\cE^*$ is $a$-regular,
and the complex constructed from $\cE$ as in Proposition \ref{linear monad} has supernatural
cohomology.
\end{proposition}

\begin{proof} By Theorem \ref{supernatural} the dual bundle $\cE^*$ is $a$-regular. Indeed,  $$h^i \cE^*(a-i)=h^{m-i} \cE(i-a-m-1)=0 \hbox{ for } i\ge 1, $$because $i-a-m-1\le i-1+z_{m-1}\le z_{m-i}$.
\end{proof}

\section{Proof of the \bs Conjectures}
\label{main proof}

Both Theorems \ref{existence0} and \ref{main0}  reduce at once to the case
$c=n$ of modules of finite length. To simplify notation we will work in that context.

\begin{theorem}\label{main} The cone defined by the upper facet equations
contains the Betti tables of 
minimal free resolutions of all finitely generated graded $S$-modules.
\end{theorem}

\begin{proof}
Suppose $f=(f_0<f_1<\cdots < f_n)$ is a degree sequence and $\tau$
is an integer with $0 \le \tau \le n-1$ such that $f_{\tau+1}=f_\tau+2$,
so that $\facet(f,\tau)$ is defined.
By Theorem \ref{supernaturals exist} there is a vector bundle $\cE$ on $\PP^{n-1}$
with supernatural cohomology whose Hilbert polynomial $\chi(\cE(d))$ has roots
$$
(z_1>\cdots >z_{n-1})=(-f_0>\cdots>-f_{\tau-1}>-f_{\tau+2} >\cdots>-f_n).
$$
Let 
$$
E:\quad 0\to E^0\to\cdots\to E^{m}\to 0
$$
be the linear complex made from $\cE$ with $a=f_n-n+1$ as in Proposition \ref{linear monad}, so that
$E^0=S(a)^{b_0}$. 
Set $c=f_\tau$. 
Theorem \ref{modified} proves that  $\langle -, E\rangle_{c,\tau}$ is non-negative
on the cone of Betti tables of all minimal resolutions. Thus it suffices
to prove that $ \facet(f,\tau)$ is defined by the vanishing
of this functional.

By construction, the coefficient of $\beta_{i,j}$ for $j \ge f^+_i$ in $\langle -, E\rangle_{c,\tau}$ is
zero, so $ \langle F,E \rangle_{c,\tau}=0$ for all resolutions $F$ 
such that $\beta_{i,j}(F)=0$ for all $i,j$ with  $j < f^+_i$.
On the other hand we have 
$$\langle F,E\rangle_{c,\tau}=\langle F,E \rangle$$
for a resolution $F$ of a module $M$ such that $\beta_{i,j}(M)=0$ for
all $i,j$ with $j > f^-_i$. 
For these modules we check the vanishing criterion of Theorem \ref{positivity}.
We have 
$$
\reg M \le f_n-n\;  \hbox{ and } \; \reg E_0=-f_n+n-1
$$ 
so the first inequality in Theorem \ref{positivity} is satisfied. 
For the remaining
inequality, note that if $j>0$ then
$$
\reg F_{j-1} \le f_{j-1} \; \hbox{ and }\reg H^j(E)=z_{j-1}-1.
$$
Thus $\langle -, E\rangle_{c,\tau}$ vanishes on $\facet(f,\tau)$.
Finally we observe that $\langle F, E\rangle \not=0$ for the pure complex $F$ with degree sequence $f$.
We conclude that  $\langle -, E\rangle_{c,\tau}$ is the supporting equation of this facet. 
\end{proof}

\begin{proof}[Proof of Theorem \ref{main0}]
As discussed in \S \ref{BS explanation} the 
 \bs fan is simplicial, so every Betti table in it is in a unique minimal
cone defined by a chain of degree sequences.
The result follows from  Theorem \ref{main}.
\end{proof}

We return to the Example \ref{Ex1} of 
 $\facet((-4,-3, 0,2,4,6,7,9), 3)$. The equation of the supporting hyperplane,
 computed before, is given by the dot
 product with the matrix
\scriptsize
$$
\begin{matrix}
\vdots&\vdots&\vdots&\vdots&\vdots& \vdots &\vdots& \vdots &\cr
1755&-385&0& 0&66& -70 & 0&100&\cr
385&0& 0&-66& 70 & 0&-100&175&\cr
0^*&0^*&66& -70& 0&100&-175&189&\cr
0&0&70&0& -100 & 175 &-189&140&\cr
0&0&0^*&100& -175 & 189&-140&60&\cr
0&0& 0&175&-189&140& -60&0&\;  \cr
0&0& 0 &0^+&0^+&60 &0&0 & \;0\hbox{-th row} \cr
0&0& 0 &0& 0 &  0^*&0^*&44&\cr
0&0& 0 &0& 0 &  0&0&0^*&\cr
0&0& 0 &0& 0 &  0&0&0&\cr
\vdots&\vdots&\vdots&\vdots& \vdots &\vdots& \vdots& \vdots \cr
\end{matrix}
$$
\normalsize
Now we can interpret all the  entries of this table: Suppose there is vector bundle $\cE$ on $\PP^6$ with natural cohomology for $(z_1 > \cdots >z_6)=(4,3,0,-6,-7,-9)$. The coefficients of the facet equation, up to a common factor, are the dimensions of cohomology groups of $\cE$. The rank of the bundle $\cE$ is at least $15$ by Proposition \ref{supernatural rank bound}, and if such bundle of rank $15$ exists, then the integral factor above is $1$. In this case the monad $E$ for $\cE$ would have shape
$$
0 \to S(3)^{455} \to S(4)^{1260}\to S(5)^{1480}\to S(6)^{924} \to S(7)^{308} \to S(8)^{44} \to 0
$$
with homology  $H^i(E)$ only in (cohomological) degrees $i=0,2,3,5$. 
Indeed the Hilbert series $H_E(t)$ of such a complex would be
{\small
\begin{eqnarray*} H_E(t)& =&
(-44t^{-8}+308t^{-7}-924t^{-6}+1480t^{-5}-1260t^{-4}+455t^{-3})/(1-t)^7 \cr
&=& -(44t^{-8})-(60t^{-5}+140t^{-4}+189t^{-3}+175t^{-2}+100t^{-1})+\cr
&&(70t^1+66t^2)+(385t^5+1755t^6 + \cdots). 
\end{eqnarray*}
}
We do not know whether there is such a bundle with rank
as small as $15$. 
The bundle provided by Theorem \ref{supernaturals exist} has  rank ${6\choose {2,1,2,1}}=180=12*15$. The complex $E$ that corresponds to it thus has Betti numbers
$12$ times larger than the hypothetical monad indicated above.
The homogeneous bundle  of Theorem \ref{supernaturals in char 0 exist}   has rank $3^5*5*7*11*13=1216215$ in this case. Hence in characteristic 0, any sufficiently large multiple of $3*15=45$  occurs as a rank for the root sequence $(4,3,0,-6,-7,-9)$. 
In Remark \ref{supernatural ranks} we conjectured that any sufficiently high multiple of $15$  occurs as a rank.

\section{The Cone of Cohomology Tables}
\label{vector bundles}
Throughout this section we set $m=n-1$ and consider coherent sheaves and vector
bundles on $\PP^m=\PP^m_K$.

Let $\cE$  be a vector bundle on $\PP^m$. 
By Serre vanishing and Serre duality only the 0-th and m-th row
of the cohomology table of $\cE$ can have infinitely many non-zero entries.
We define the \emph{cohomology range} 
\begin{align*}
 r(\cE)&=(r_1(\cE)\ge \cdots \ge  r_{m+1}(\cE)=-\infty )\cr
R(\cE) &=(\infty= R_0(\cE)\ge \cdots \ge  R_{m}(\cE)) 
\end{align*}
  of  $\cE$ by the formulas
\begin{align*}
r_i(\cE) &= \sup\{d \mid H^j(\cE(d-j)=0 \hbox{ for all  } j < i \},\cr
R_i(\cE) &= \inf\{d \mid H^j(\cE(d-j)=0 \hbox{ for all  } j\ge  i \}.
\end{align*}
For example, $R_1(\cE)$ is the Castelnuovo-Mumford regularity of $\cE$.
It follows from the definition that 
$h^i \cE(d-i)=0$ except possibly when  $r_{i+1}(\cE) < d < R_i(\cE)$, whence
the name ``cohomology range''.

\begin{proposition}\label{range}
If $\cE$ is a supernatural vector bundle on $\PP^m$ then
$$ 
\max\{i \mid h^i\cE(d-i)\not=0 \} =\min\{i \mid h^i\cE(d-i)\not=0 \} 
$$
and
$$r_i(\cE)=R_{i}(\cE)= z_i(\cE)+i.$$
\end{proposition}
\begin{proof} Immediate from Theorem \ref{supernatural}.
\end{proof}

We now define the \emph{fan of supernatural cohomology tables} inside
$\CC$. We work with root sequences $z$ of a given length $m=n-1$.
First, we define the \emph{supernatural cohomology table of the root sequence 
$z=(z_1>\cdots>z_m)$} by the formula
$$ \gamma_{i,d}(z)=
\begin{cases}
\frac{1}{m!} \prod_{j=0}^{m-1} |d-z_j|  & \hbox{for } z_{i}> d > z_{i+1} \cr
 0 & \hbox {otherwise.}\cr
 \end{cases}
 $$ 
 We
partially order these supernatural cohomology tables 
by the termwise partial order on their root sequences
$$
z=(z_1>\cdots >z_{m}) \ge z' =(z_1'>\cdots >z_{m}') \Leftrightarrow  z_i \ge z_i' \hbox{ for all } i.
$$ 
This coincides with the other obvious way of defining a partial order:

\begin{proposition} If $\cE$ and $\cE'$ are supernatural bundles then
$z(\cE)\ge z(\cE')$ if and only if
$$
\max\{i \mid h^i\cE(d-i)\not=0 \}  \ge \max\{i \mid h^i\cE'(d-i)\not=0 \}  \hbox{ for all } d.
$$
\end{proposition}

\begin{proof} If $\cE$ is supernatural then  $\max\{i \mid h^i\cE(d-i)\not=0 \} =\min\{i \mid h^i\cE(d-i)\not=0 \}$ by Proposition \ref{range}. The result follows.
\end{proof}

In a bounded range of root sequences, say $z$ with $a \ge z_1 > \cdots >z_{m} \ge b$ all maximal chains 
of root sequences have the same  length. Hence the order complex
of chains ($=$ totally ordered subsets) is equidimensional. The simplices defined by a chain of supernatural cohomology tables intersect only in common faces. This follows by the same argument as for pure resolutions of Cohen-Macaulay modules given in \bs \cite{BS}, Proposition 2.9.
Thus the supernatural cohomology tables of root sequences in a bounded range define a geometric realization of the order complex of such root sequences, and this
defines the \emph{fan of cohomology tables}. Theorem \ref {main1} which we will prove below, implies the convexity of this fan.

Much of Proposition \ref{outer facets} remains true for the fan of cohomology tables. However, the outer facets that are not defined by a single $\gamma_{i,d}$, correspond this time to a sequence $z^+> z > z^-$, which differ only in one position, say $i $, and $z^+_i-1=z_i=z^-_i+1$. A consequence of Theorem \ref{cohomology facet} below is that the equation of the outer facet coincides for all maximal chains containing  the sequence $z^+>z>z^-$.

Let $F$ be a minimal free resolution of a graded $S$-module of finite length.
In the following we set 
$$
 \langle F, \cE\rangle:=  
 \langle F, E\rangle
$$
where $E=E_a$ is the dual of the free resolution of 
$\oplus_{d\ge a} H^0\cE^*(d)$ for $a \gg 0$ as in Proposition \ref{linear monad}.
In case of a pure resolution $F$ with degree sequence $f$ the formula simplifies to 
$$\langle F, \cE\rangle=  \sum_{j \le i}
 (-1)^{i-j} \beta_{i,f_i} h^j(\cE(-f_i))
$$

\begin{example}\label{7.3} The non-zero cohomology groups for the
three root sequences $z^+=( 3,1,-4),z=(3,0,-4)$ and $z^-=(3,-1,-4)$ are indicated in the following diagram
\scriptsize %\small
$$
\begin{matrix}
\cdots \cr
\cdots \cr
\cdots \cr
\cdots \cr
\end{matrix}\quad
\begin{matrix}
 *&*&*&\cdot &\cdot & \cdot  & \cdot  &\cdot  &\cdot &\cdot \cr
 \cdot &\cdot & \cdot &*&*&+&+&\cdot  & \cdot   &\cdot \cr
 \cdot &\cdot &\cdot &\cdot & \cdot  &-&-&*&\cdot &\cdot   \cr
\cdot &\cdot &\cdot &\cdot & \cdot  &\cdot & \cdot  &\cdot & *& *   \cr
\end{matrix}\quad
\begin{matrix}
\cdots \cr
\cdots \cr
\cdots \cr
\cdots \cr
\end{matrix}\;
$$
\normalsize
We can find the equation for the supporting hyperplane of a facet using an algorithm that
is completely analogous to that of Proposition \ref{upper equation}.
Applying it in our example, we obtain the linear form whose coefficients
are indicated in the following table:
\scriptsize %\small
$$
\begin{matrix}
\cdots \cr
\cdots \cr
\cdots \cr
\cdots \cr
\end{matrix}\quad
\begin{matrix}
 0^*&0^*&0^*&0&0& 0 & 0 &0 &0&0\cr
 0&0& 2&0^*&0^*&0^+&0^+&0 & 0  &0\cr
 0&-2&0&0& 35 &-70&42&0^*&0&0  \cr
2&0&0&-35& 70 &-42& 0 &5& 0^*& 0^*  \cr
\end{matrix}\quad
\begin{matrix}
\cdots \cr
\cdots \cr
\cdots \cr
\cdots \cr
\end{matrix}\;
$$
\normalsize
The shape of the facet equation leads us  to consider a pure complex
$$ F: 0 \to S^2(-4) \to S^{35}(-1) \to S^{70} \to S^{42}(1) \to S^5(3) \to 0,$$
corresponding to the degree sequence $f$ that is the negative of the union of the root sequences
$z^+, z$ and $z^-$.

The coefficients of $\langle F, -\rangle$ are
\scriptsize %\small
$$
\begin{matrix}
\cdots \cr
\cdots \cr
\cdots \cr
\cdots \cr
\end{matrix}\quad
\begin{matrix}
 0^*&0^*&0^*&-2&0& 0 &35  &0&0\cr
 0&0& 2&0^*&0^*&-35&70&0 & 0  &0\cr
 0&-2&0&0& 35 &-70&42&0^*&0&0  \cr
2&0&0&-35& 70 &-42& 0 &5& 0^*& 0^*  \cr
\end{matrix}\quad
\begin{matrix}
\cdots \cr
\cdots \cr
\cdots \cr
\cdots \cr
\end{matrix}\;
$$
\normalsize
These coefficients coincide with those of the facet equation except  that
in the facet equation some terms in the top two rows  
have been replaced by zeros. This is exactly the effect of the modification
described in \S \ref{pos2}.
\end{example}

\begin{theorem}\label{cohomology facet}
Suppose that the three root sequences $z^+> z >z^-$ differ only in the 
$i$-th spot, where they have $z^+_i-1= z_i = z_i^- +1$. 
Let $f$ be the degree sequence
that is the negative of the union of these three root sequences, and let $\tau$ be such that
$f_\tau = -z_i,$ the middle value. Let $ c=f_{\tau -1}=f_\tau -1$, the smaller value.
Let $F$ be a pure resolution corresponding to the degree sequence $f$. 
The functional $ \langle F, \cE\rangle_{c,\tau}$ is positive on the supernatural
bundle $\cE$ with root sequence $z$, and vanishes on all supernatural bundles
$\cE$ with root sequence $\le z^-$ or $\ge z^+$. 
\end{theorem}

\begin{proof}
The functional coming from a pure resolution $F$ with degree sequence 
$f$ is
\begin{eqnarray*}
 \langle F, \cE\rangle_{c,\tau}&= & \sum_{j=i< \tau}  (-1)^{i-j} \beta_{i,f_i} h^j\cE(-f_i) \cr
 &+& \sum_{j=i-1< \tau}  (-1)^{i-j} \beta_{i,f_i} h^j\cE(-f_i) \cr
&+& \sum_{j \le i-2} (-1)^{i-j} \beta_{i,f_i} h^j\cE(-f_i))\cr
\end{eqnarray*}
It vanishes trivially
on all supernatural bundles with degree sequence $\ge z^+$ and is positive on the one with root sequence $z$ because the coefficients of the diagonal terms $i=j$ are positive. 
If $\cE$ is a supernatural bundle with root sequence $\leq z^-$ then 
$
\langle F, \cE\rangle_{c,\tau}= \langle F, \cE\rangle.
$
To prove that this is zero, we check the conditions of Theorem \ref{positivity}:

The module $M=\coker(F_1 \to F_0)$ has regularity $f_n-n$. The module $E^0$ has regularity
$\reg E^0=z_{m}(\cE)+m=w_{m}+m \le z_{m}^- +m \le z_{m}^++m=-f_n+m$. Thus
$$
\reg M + \reg E^0\le f_n-n-f_n+m=-1 <0.
$$ 
Moreover for $j>0$, we have $\reg H^j(E)=w_{j}-1 \le z^-_{j}-1\le -f_{j-1}-1$ and $\reg F_{j-1}=f_{j-1}$,
hence
$$
\reg F_{j-1} + \reg H^j(E) \le-1<0.   %\qquad\qquad\qedhere
$$ 
\vskip -25pt
\end{proof}

{\it Proof} of Theorem \ref{main1}.
By Theorem \ref{cohomology facet} the 
equation of the non-trivial outer facets are given by functionals
$\langle F, -\rangle_{c,\tau}$
for suitable pure free resolutions $F$ and integers $c,\tau$.
By Theorem \ref{modified}, this functional
is non-negative on the monad $E=E_a$ obtained from the free resolution 
of $\oplus_{d\ge a} H^0 \cE^*(d)$ with $a \gg 0$
for any vector bundle $\cE$. 
\qed
\medskip

{\it Proof} of Corollary \ref{slope}. By restriction to a hyperplane
and induction we see that the Hilbert polynomial of $\cE$ has the form
$$
p_\cE(t) = (\rank \cE){t+m \choose m}+\deg \cE {t+m-1 \choose m-1}+O(t^{m-2}).
$$
Thus
$$
\frac{p_\cE(t)}{\rank \cE}=\frac{t^m}{m!}
+\biggl( \frac{1}{m!}{m+1\choose 2}+\frac{\mu(\cE)}{(m-1)!}\biggr) t^{m-1}+O(t^{m-2}),
$$
and in particular
$$
\mu(\cE) = -\frac{1}{m}\sum_{i=1}^{m} (z_i+i),
$$
where $z_1,\dots,z_{m}$ are the roots of $p_\cE$.

The cohomology table of $\cE$ is a rational combination of cohomology tables of vector bundles $\cE_k$ with supernatural cohomology. Hence the normalized Hilbert polynomial of $\cE$
may be written in the form
$$
\frac{p_\cE(t)}{\rank \cE} = \sum_k \lambda_k  \frac{p_{\cE_k}(t)}{\rank \cE_k}
$$
for some $\lambda_k>0$. Comparing coefficients of $t^m$ and $t^{m-1}$
we obtain
%
%\begin{eqnarray*}
%\sum_k \lambda _k &=&1,\hbox{ and}\cr
%\sum_k \lambda_k\mu(\cE_k) &=& \mu(\cE).
%\end{eqnarray*}
%
$\sum_k \lambda _k =1$ 
and
$\sum_k \lambda_k\mu(\cE_k) = \mu(\cE).$

On the other hand by Theorem \ref{supernatural}, the zeroes of the $\cE_k$ satisfy
$$
r_i(\cE) \le r_i(\cE_k) = z_i(\cE_k)+i = R_{i}(\cE_k) \le R_{i}(\cE),
$$
because otherwise $\cE_k$ would contribute to a cohomology group outside the range of $\cE$.
From the expression for $\mu(\cE_k)$ in terms of $z_i(\cE_k)$ we get
$$
 \frac{1}{m}\sum_{i=1}^{m}  r_i(\cE)\le -\mu(\cE_k) \le \frac{1}{m}\sum_{i=1}^{m} R_i(\cE)
 $$
Multiplying this inequality by $\lambda_k$ and summing over $k$ we obtain
the assertion of the Corollary.
\qed

\bigskip

\vbox{\noindent Author Addresses:\par
\smallskip
\noindent{David Eisenbud}\par
\noindent{Department of Mathematics, University of California, Berkeley,
Berkeley CA 94720}\par
\noindent{eisenbud@math.berkeley.edu}\par
\smallskip
\noindent{Frank-Olaf Schreyer}\par
\noindent{Mathematik und Informatik, Universit\"at des Saarlandes, Campus E2 4, 
D-66123 Saarbr\"ucken, Germany}\par
\noindent{schreyer@math.uni-sb.de}\par
}

\end{document}

%% file: preamble.tex
\def\antiddot{\mathinner{\mkern1mu\raise1pt\vbox{\kern7pt\hbox{.}}\mkern2mu
        \raise4pt\hbox{.}\mkern2mu\raise7pt\hbox{.}\mkern1mu}}

\newcommand{\BB}{{\mathbb B}}
\newcommand{\CC}{{\mathbb C}}

\newcommand{\PP}{{\mathbb P}}
\newcommand{\QQ}{{\mathbb Q}}

\newcommand{\TT}{{\mathbb T}}

\newcommand{\coker}{{\rm{coker}\,}}

\newcommand{\s}{\mathcal}

\newcommand{\cO}{{\s O}}

\newcommand{\cL}{{\s L}}

\newcommand{\tensor}{\otimes}

\newcommand{\punkt}{\hspace{-.3ex}\raise.15ex\hbox to1ex{\Huge.}}

\def \fix#1 {{\hfill\break \bf (( #1 ))\hfill\break}}

\DeclareMathOperator{\reg}{reg}

\DeclareMathOperator{\rank}{rank}

\newtheorem{theorem}{Theorem}[section]
\newtheorem{lemma}[theorem]{Lemma}
\newtheorem{proposition}[theorem]{Proposition}
\newtheorem{corollary}[theorem]{Corollary}

\theoremstyle{definition}

\newtheorem{remark}[theorem]{Remark}

\newtheorem{example}[theorem]{Example}

%% file: 080712-BoijSoderbergConj.bbl
\begin{thebibliography}{ABC99}

%\bibitem[1965]{BR} D. Buchsbaum and D.S.Rim.
%A generalized Koszul complex. III. A remark on generic acyclicity.
%Proc. Amer. Math. Soc. 16 (1965) 555--558. 

\bibitem[2006]{BS} M. Boij and J. S\"oderberg. Graded betti numbers of Cohen-Macaulay modules and the multiplicity conjecture. math.AC/0611081.

\bibitem[2008]{BS2}M. Boij and J. S\"oderberg.
Betti numbers of graded modules and the Multiplicity Conjecture in the non-Cohen-Macaulay case. 
Preprint: arXiv:0803.1645.

\bibitem[1973a]{BE} D. Buchsbaum and D. Eisenbud. Remarks on ideals and resolutions.
Symposia Math. XI (1973) 193--204.

\bibitem[1973b]{WMACE}D.~A.~Buchsbaum and D.~Eisenbud. What makes a complex exact?  J. Algebra  25  (1973) 259--268.

\bibitem[1995]{Eis} D. Eisenbud. {\bf Commutative Algebra with a View Toward
Algebraic Geometry}. Graduate Texts in Math. 150, Springer-Verlag New York (1995).

\bibitem[2003]{EFS} D. Eisenbud, G.~Fl\o ystad and F.-O. Schreyer.  
Sheaf cohomology and free resolutions over exterior algebras.
 Trans. Amer. Math. Soc.  355  (2003) 4397--4426.

\bibitem[2007]{EFW} D. Eisenbud, G. Fl\o ystad and J. Weyman. The existence of pure
free resolutions. arXiv:0709.1529.

%\bibitem[2006]{EHU} D. Eisenbud, C. Huneke and B. Ulrich. 
%The regularity of Tor and graded Betti numbers.
%Amer. J. Math. 128 (2006) 573--605.

\bibitem[2003]{ES-C} D. Eisenbud and F.-O. Schreyer.
Resultants and Chow forms via exterior syzygies.
 J. Amer. Math. Soc.  16  (2003) 537--579.

\bibitem[2007]{Erman} D. Erman. The Semigroup of Betti Diagrams.
arXiv:0806.4401 

\bibitem[2007]{Francisco-Srinivasan} 
C.A Francisco and H. Srinivasan. Multiplicity conjectures.
in {\bf Syzygies and Hilbert Functions}, ed. I. Peeva, Lect. Notes in 
Pure and Appl. Math., Chapman and Hall, NY, 2007

\bibitem[M2]{M2}
D. R. Grayson and Stillman, M. E. Stillman.
Macaulay 2, a software system for research
in algebraic geometry.  Available at http://www.math.uiuc.edu/Macaulay2/

\bibitem[1982]{VB} R. Hartshorne and A. Hirschowitz.
Cohomology of a general instanton bundle.
Ann. Sci. de l'\'Ecole Normale Sup. (1982) 365--390.

\bibitem[1984]{HK} J. Herzog and M. K\"uhl. On the betti numbers of finite pure and linear
resolutions. Comm. in Alg. 12 (13) (1984) 1627--1646.

\bibitem[1998]{HS} J. Herzog and H. Srinivasan. Bounds for Multiplicities.
Trans. Am. Math. Soc. (1998) 2879--2902.

%\bibitem [1985]{HM} C. Huneke and M. Miller. A note on the multiplicity of 
%Cohen-Macaulay algebras with pure resolutions. Can J. Math (1985) 1149--1162.

\bibitem[1974]{Kirby} D. Kirby. A sequence of complexes associated to a matrix. 
J. London Math. Soc.
(1974) 523--530.

\bibitem[2008]{kunte} M.~Kunte.
Gorenstein modules of finite length. 
Thesis, Uni. des Saarlandes (2008).

\bibitem[1992]{Mukai} S. Mukai. Curves and symmetric spaces.
Proc. Japan Acad. Ser. A Math. Sci. 68 (1992), no. 1, 7--10.

\bibitem[2003]{Mukai-unpublished} S. Mukai.  
Curves and Symmetric Spaces II. RIMS preprint 2003,
http://www.kurims.kyoto-u.ac.jp/preprint/preprint\_y2003.html

%\bibitem[1974]{PS} C. Peskine and L. Szpiro. C. R. Acad. Sc. Paris, t.278 Series A
%(1974) 1421--1424.

\bibitem[1973]{PS} C. Peskine and L. Szpiro.
Dimension projective finie et cohomologie locale. Applications ˆ la dŽmonstration de conjectures de M. Auslander, H. Bass et A. Grothendieck. 
Inst. Hautes ƒtudes Sci. Publ. Math. 42 (1973) 47--119. 

\bibitem[1986]{Schreyer} F.-O. Schreyer. 
Syzygies of canonical curves and special linear series.
 Math. Ann.  275  (1986),  no. 1, 105--137.
 
 \bibitem[2006]{JS} J. S\"oderberg. 
 Graded Betti numbers and $h$-vectors of 
 level modules. Preprint, arxiv:math.AC/0612047.
 
\bibitem[2003]{W} J. Weyman. {\bf Cohomology of Vector Bundles and Syzygies}. 
Cambridge Tracts in Math. 149 (2003).

\end{thebibliography}
